\documentclass{article}
\usepackage{latexsym}
\usepackage{leftidx}
\usepackage{slashed}
\usepackage{amsmath,amsthm,pdfpages,color,pdfsync}
\usepackage{enumerate}
\usepackage{amssymb}
\usepackage{upgreek}

\usepackage[margin=1.3in,dvips]{geometry}

\def\l{\langle}
\def\r{\rangle}

\def\N{{\bf N}}

\newcommand{\vol}{\textnormal{vol}}

\def\f12{\frac 1 2}
\def\L{L}
\def\Lu{\underline{L}}

\def\ga{\gamma}

\def\ep{\epsilon}

\def\La{\Lambda}

\def\om{\omega}

\def\N{\mathcal{N}}

\def\Lb{\underline{L}}

\def\pa{\partial}
\def\les{\lesssim}

\def\cD{\mathcal{D}}
\def\R{\mathbb{R}}
\newcommand{\D}{\mbox{$D \mkern-13mu /$\,}}

\newcommand{\lap}{\mbox{$\Delta \mkern-13mu /$\,}}

\newcommand{\nabb}{\mbox{$\nabla \mkern-13mu /$\,}}
\newcommand{\J}{\mbox{$J \mkern-13mu /$\,}}

\newtheorem{thm}{Theorem}
\newtheorem{Thm}{Theorem}[section]

\newtheorem{Prop}{Proposition}[section]

\newtheorem{Lem}{Lemma}[section]

\newtheorem{remark}{Remark}
\newtheorem{Remark}{Remark}[section]

\begin{document}

\title{Asymptotic decay for the Chern-Simons-Higgs equations}

\author{Dongyi Wei \and Shiwu Yang}

\AtEndDocument{\bigskip{\footnotesize%
  \addvspace{\medskipamount}
  \textsc{School of Mathematical Sciences, Peking University, Beijing, China} \par
  \textit{E-mail address}: \texttt{jnwdyi@pku.edu.cn} \par

  \addvspace{\medskipamount}
  \textsc{Beijing International Center for Mathematical Research, Peking University, Beijing, China} \par
  \textit{E-mail address}: \texttt{shiwuyang@math.pku.edu.cn}
  }}

\date{}

\maketitle

\begin{abstract}
In this paper, we study the long time asymptotic behaviors for solutions to the Chern-Simons-Higgs equation with a pure power defocusing nonlinearity. We obtain quantitative inverse polynomial time decay for the potential energy for all data with finite conformal energy. Consequently, the solution   decays in time in the pointwise sense for all power. We also show that  for sufficiently large power  the solution decays as quickly as linear waves. Key ingredients for the  proof include vector field method, conformal compactification and the geometric bilinear trace theorem for null hypersurface developed by Klainerman-Rodnianski.

\end{abstract}

\section{Introduction}
Let $A$ be a 1-form in $\mathbb{R}^{1+2}$ with the standard Cartesian coordinates $(t, x)=(t, x_1, x_2)$. Define the associate covariant derivative 
\begin{align*}
D_\mu=\pa_{\mu}+i A_{\mu},
\end{align*}
where the Greek indice $\mu$ runs from $0$ to $2$ and $\pa_0$ stands for the derivative $\pa_t$. The commutator of this covariant derivative gives the curvature 2-form 
\begin{align*}
F_{\mu\nu}=-i[D_\mu, D_\nu]=(dA)_{\mu\nu}=\pa_\mu A_\nu-\pa_\nu A_\mu. 
\end{align*}
The  Chern-Simons-Higgs equations are a system for the connection field $A$ and scalar field $\phi$
\begin{equation}
  \label{eq:CSH:p:2d}
  \begin{cases}
  \Box_A\phi=D^\mu D_\mu \phi=|\phi|^{p-1}\phi,\\
  F_{\mu\nu}=\varepsilon_{\mu\nu\ga}J^{\gamma}[\phi],\quad J_{\gamma}[\phi]=\Im(\phi \cdot \overline{D_\ga\phi}),\\
  \phi(0, x)=\phi_0(x),\quad D_0\phi(0, x)=\phi_1(x),
  \end{cases}
  \end{equation}
  in which $\varepsilon_{\mu\nu\ga}$ is the antisymmetric volume element such that  $\varepsilon_{012}=1$ and $p>1$ is a constant. Raising or lowering the indices is with respect to the flat Minkowski metric $m_{\mu\nu}=diag\{-1, 1, 1\}$ on $\mathbb{R}^{1+2}$.

The Chern-Simons theory plays an important role  in planar physics and is completely different from the classical Maxwell theory in space dimension two. This new theory draws extensive attention both from physists and mathematicians for its rich structure and  practical application to certain planar condensed matter phenomena. For more general models and discussions related to this theory, we refer to \cite{Girvin05:QHE}, \cite{Dunne99:CSH}, \cite{Taubes80:Vort} and references therein.

 The Chern-Simons-Higgs model was first proposed in  \cite{Hong90:CSH}, \cite{Jackiw90:CSH}. The potential could be more general to allow the existence of nontrivial vortex solutions 
\cite{Taran96:CSH}, \cite{Taran99:CSH}, \cite{Spruck95:CSH}, which is closely related to high-temperature superconductivity, quantum Hall effect and the anyon in physics. 
The current study mainly focuses on the long time dynamics for solutions to the above equation, which has been studied extensively in the past decades. On one hand, as a model in Chern-Simons theory, global classical solution for sufficiently smooth initial data was first shown by Chae-Choe in \cite{Chae02:CSH:2d}, which is based on the observation that  the  solution never blows up in finite time. Another perspective to conclude this global existence result is to construct local in time solution in energy space since the energy is conserved. Such local well-posedness results under various gauges for data with lower regularity  could be found for example in  \cite{Huh05:CSH}, \cite{Huh07:CSH}, 
\cite{Huh11:CSH:energy}, 
\cite{Selberg13:CSH:low},
\cite{Okamoto13:CSD:2d}, 
 \cite{Huh14:CSH:1d},
\cite{Oh16:CSH:low}, 
\cite{Pecher16:CSH},  
\cite{yuan17:CSH},  
 \cite{Cho21:CSH:low}. These results  heavily rely on the null structure of the system and the choice of the gauge.

 On the other hand, the pure power nonlinearity model problem studied here is the geometric generalization of the defocusing semilinear wave equation in $\R^{1+2}$ with vanishing connection field $A$. Even for this simpler model with defocusing nonlinearity, plenty of questions regarding the long time dynamics of  the solutions remain unclear, particularly for the small power case. A typical question is whether the solution approaches or at least decays like linear solutions, which usually requires a priori dispersive estimate for the nonlinearity.  For such large data problem, the decay mechanism  lies behind the conformal symmetry of the equation, which was first used by Strauss in  \cite{Strauss:NLW:decay} and later by Ginibre-Velo in  \cite{velo85:global:sol:NLW}, \cite{Velo87:decay:NLW}  to study the long time behaviors for solutions to the defocusing semilinear wave equations with superconformal power. To go beyond the superconformal power,  Glassey and Pecher in \cite{Pecher82:decay:3d},
  \cite{Pecher82:NLW:2d} observed that the potential energy still decays in time by using  Grownwall's inequality. The $r$-weighted energy method originally introduced by Dafermos-Rodnianski in \cite{newapp} has been successfully applied to this problem by the authors in  \cite{yang:scattering:NLW},  \cite{yang:NLW:ptdecay:3D} and new multipliers were introduced 
  in \cite{tao12:1d:NLW}, \cite{Yang:NLW:1d}, \cite{yang:NLW:2D} for the lower dimensional case. 

  There are not too many results regarding the long time behaviors for solutions to the Chern-Simons-Higgs equations. Huh in \cite{Huh17:CS:conformal}  obtained potential energy decay for the above model \eqref{eq:CSH:p:2d} by using the method developed by Glassey-Pecher in  \cite{Pecher82:NLW:2d}. Small data global existence with quantitative decay estimates has been shown by Chae-Oh  in \cite{Oh17:CSH:2d:small}, in which the system  can be view as coupled nonlinear wave and Klein-Gordon equations. See recent advances and discussions  for example in \cite{Delort04:NLW:2D}, \cite{Yang:mMKG}, \cite{Dong23:2d:WKG}, \cite{Ionescu20:EKG}. 
The problem becomes much harder in space dimension one as linear wave does not decay.  
For such geometric nonlinear wave equations, one can show that the solution grows at most polynomially \cite{Huh20:MKG:1d}, \cite{Huh21:CSH:Longtime}.

The purpose of this work is to study the long time dynamics for solutions to the Chern-Simons-Higgs equation \eqref{eq:CSH:p:2d} for general large data. We extend the decay estimates obtained in  \cite{yang:NLW:2D} for the flat case when $A=0$ to solutions of \eqref{eq:CSH:p:2d}. The main new challenge is that the rough decay estimate of the solution due to the small power of nonlinearity is not sufficient to control the higher order nonlinear terms arising from the connection field $A$, for which we will exploit the special null structure of the system.

To state our main result, for  $0\leq \ga\leq 2$, define the gauge invariant weighted energy norm of the initial data
\begin{align*}
  \mathcal{E}_{k,\ga}=\sum\limits_{l\leq k}\int_{\mathbb{R}^2}(1+|x|)^{\ga+2 l}(|\bar{D}^{l+1}\phi_0|^2+|\bar{D}^l \phi_1|^2)+(1+|x|)^{\ga}|\phi_0|^{p+1}dx.
\end{align*}
Here $\bar{D}$ is short for $(D_{1}, D_{2})$. The above gauge invariant norm can be defined once we choose a representation (for example a particular gauge like Lorenz gauge) of the connection field $A$ on the initial hypersurface $\{t=0\}$, which follows by noting that  the curvature 2-form $F=dA$ is given by $\phi_0$ and $\phi_1$. In this paper we will only use the initial conformal energy $\mathcal{E}_{0, 2}$ and the first order energy $\mathcal{E}_{1, 0}$.

Our first result is the time decay of the potential energy.
\begin{thm}
  \label{thm:main:td}
  Consider the Cauchy problem to the Chern-Simons-Higgs equation \eqref{eq:CSH:p:2d} in $\mathbb{R}^{1+2}$ with a pure power defocusing nonlinearity and finite initial conformal energy $\mathcal{E}_{0, 2}$. Then the potential energy decays inverse polynomially in time 
  \begin{align*}
    \int_{\mathbb{R}^2}|\phi(t, x)|^{p+1} dx\leq C\mathcal{E}_{0, 2} (1+t)^{- \min\{\frac{p-1}{2}, 2 \}}
    ,\quad \forall t\geq 0
  \end{align*}
  for some constant $C$ depending only on $p$.
\end{thm}
\begin{Remark}
A weaker decay estimate 
\begin{align*}
    \int_{\mathbb{R}^2}|\phi(t, x)|^{p+1} dx\leq C\mathcal{E}_{0, 2}(1+t)^{3-p},\quad p<5
    ,\quad \forall t\geq 0.
  \end{align*}
  had been obtained by Huh in \cite{Huh17:CS:conformal}, based on the method developed by Pecher in  \cite{Pecher82:NLW:2d} for the study of time decay for the solution with vanishing connection field $A$.  
\end{Remark}
The above time decay of the potential energy together with a type of  logarithmic Sobolev embedding is sufficient to conclude the pointwise decay estimates for the solution.
\begin{thm}
\label{thm:main:pd:gen}
If the initial data $(\phi_0, \phi_1)$ for the Chern-Simons-Higgs system \eqref{eq:CSH:p:2d} are bounded in $\mathcal{E}_{0, 2}$ and $\mathcal{E}_{1, 0}$, then the solution $\phi$ satisfies the rough pointwise decay estimates for all $1<p\leq 4$
\begin{equation*}
|\phi(t, x)|\leq C   (1+t)^{-\frac{p-1}{8}   } \sqrt{\ln(2+t)},\quad \forall t\geq 0
\end{equation*}
with some constant $C$ depending on  $p$, $\mathcal{E}_{0, 2}$ and $\mathcal{E}_{1, 0}$.
 
 \end{thm}
The above rough decay estimate is far from optimal. However the proof also implies that the solution decays faster away from the light cone. Moreover for sufficiently large $p$, we can show that the solution decays as fast as linear waves in $\mathbb{R}^{1+2}$.  
\begin{thm}
\label{thm:main:pd:imp}
Assume that the initial data $(\phi_0, \phi_1)$  to the system  \eqref{eq:CSH:p:2d} are supported in $\{|x|\leq 1\}$. Then for all $p>4$, the solution $\phi$  verifies the sharp time decay  
 \begin{equation*}
|\phi(t, x)|\leq C (1+t+|x|)^{-\frac{1}{2}},\quad \forall t\geq 0
\end{equation*}
for some constant $C$ depending only on $p$, $\mathcal{E}_{0, 2}$ and $\mathcal{E}_{1, 0}$.
 
 \end{thm}
 \begin{Remark} 
 Compact support assumption on the initial data is merely for simplicity. Same result holds for data bounded in some weighted energy space. For general data, inspired by the work \cite{YangYu:MKG:smooth}, one can first study the solution in the exterior region $\{|x|\geq t+R\}$. By choosing $R$ sufficiently large, the long time behavior of the solution in this region is then reduced to the classical small data problem for nonlinear wave equations verifying the null condition (\cite{Alinhac01:null:2D:1}, \cite{Alinhac01:null:2D:2}). The proof for the above result mainly concentrates on studying the solution in the interior region with arbitrarily large data.   
\end{Remark}

\begin{Remark}
For the flat case when $A=0$, it has been shown by the authors in \cite{yang:NLW:2D} that sharp time decay for the solution holds for $p>\frac{11}{3}$. We believe that our method also work for this larger range of $p$. However, the proof will be more involved. We will discuss this with more details during the proof. 
\end{Remark}

The proof for the potential energy decay, that is the main theorem \ref{thm:main:td}, is similar to that in \cite{yang:NLW:2D} by the authors for the defocusing semilinear wave equations. For  the Chern-Simons-Higgs system studied here, the key observation is that the error term 
\[
X^\nu F_{\mu\nu}J[\phi]^{\mu}=X^\nu \varepsilon_{\mu\nu \gamma}J[\phi]^{\gamma}J[\phi]^{\mu}
\]
arising from the connection field $A$ for any multiplier $X$ actually vanishes as $\varepsilon_{\mu\nu\gamma}$ is skew symmetric. This in particular indicates that any multiplier that can be used for the flat case (with vanishing connection field $A$) also works for  the general Chern-Simons-Higgs equation. For readers' interest, we briefly review the ideas for the proof. 

Instead of using the conformal Killing vector field, which is of particular importance for the superconformal case (see discussions for example in \cite{vonWahl72:decay:NLW:super}, \cite{Velo87:decay:NLW}, \cite{Pecher82:decay:3d}, \cite{Pecher82:NLW:2d}, \cite{Huh17:CS:conformal}), non-spherical symmetric vector field 
\begin{align*}
X= (x_2^2+(t-x_1)^2+1)\partial_t+(x_2^2-(t-x_1)^2)\partial_1+2(t-x_1)x_2\partial_2
\end{align*}
will be applied to the exterior region  $\{t\leq x_1\}$. This shows the weighted energy estimate through the null hyperplane $\{t=x_1\}$ 
\begin{equation*}
  \int_{t=x_1\geq 0}x_2^2|(D_t+D_1)\phi|^2(t,t,x_2) +|D_2\phi(t, t, x_2)|^2+\frac{2}{p+1}|\phi(t, t, x_2)|^{p+1} dtdx_2
\leq  C\mathcal{E}_{0, 2}.
\end{equation*}
Then in the interior region, take the multiplier  
 \begin{align*}
X=u_1^{\frac{p-1}{2}}(\partial_t-\partial_1)+u_1^{\frac{p-1}{2}-2}x_2^2(\partial_t+\partial_1)+2u_1^{\frac{p-1}{2}-1}x_2\partial_2,\quad u_1=t-x_1+1
\end{align*}
and the region $\{t\geq x_1\}$. By using the above  weighted energy bound on the null hyperplane $\{t=x_1\}$ (as boundary of the chosen region), one then can obtain that  
\begin{align*}
\int_{x_1\leq t} u_1^{\frac{p-5}{2}}(x_2^2+u_1^2)|\phi(t, x)|^{p+1} dx \leq C\mathcal{E}_{0, 2}.
\end{align*}
 Since we only study the solution in the future $t\geq 0$, restricting the integral to the half plane $\{x_1\leq 0\}$, the above estimate then implies that 
\begin{align*}
\int_{x_1\leq 0} (1+t)^{\frac{p-1}{2}} |\phi(t, x)|^{p+1} dx\leq  C\mathcal{E}_{0, 2}
\end{align*}
since $u_1=t-x_1+1$. The potential energy decay then follows by symmetry $x_1$ to $-x_1$. 

\bigskip

For the rough pointwise decay estimate for the solution,  we rely on the logarithmic Sobolev embedding as in the flat case. However, the new difficulty for the Chern-Simons-Higgs equation is the polynomial growth bound for the first order energy, that is the $H^2$ bound for the solution. To control the first order energy, one needs to commute the equation with derivatives. A typical nonlinear term arising from the connection field $A$ is of the form 
$$F_{\nu \mu}D^\nu \phi =\varepsilon_{\nu\mu \gamma} D^\nu \phi \Im(\phi\cdot \overline{D^{\gamma}\phi}) \approx |\phi| |D\phi|^2,$$
which seems not possible to be controlled by using the standard Sobolev embedding. The idea is to rely on a type of Strichartz estimate for linear half wave equation. See Lemme \ref{lem:Strichartz1} for more details. Now to obtain necessary dispersive estimate for $D\phi$, we can write the Chern-Simons-Higgs equation as first order hyperbolic system for $D\phi$ under Coulomb gauge condition. We emphasize here that the gauge condition is only used at this point.  By using Riesz transform together with the potential energy and weighted energy decay estimates for the solution, we can show that
  \begin{align*}
\|D\phi\|_{L^{4}([T,T+1];\dot{B}_{\infty,2}^{-s})}\les 1,\quad \forall T\geq 0. 
\end{align*}
This dispersive estimate is sufficient to control the above nonlinearity. Indeed by using Gagliardo-Nirenberg interpolation for Besov space and the energy conservation, we can bound that 
\begin{align*}
 \||D\phi|^2 |\phi |\|_{L^2} &   \les   \|D\phi\|_{\dot{B}_{\infty,2}^{-s}}^{ \frac{2+2\epsilon}{2+\epsilon} } \|DD\phi\|_{L^2}^{ \frac{2 s(1+\epsilon )}{2+\epsilon}  }.
\end{align*}
Here $\epsilon$ is a small positive constant and $s$ is nonnegative such that  $ s(1+\epsilon)<1$. By doing the standard energy estimate for $D\phi$, we then can show that the first order energy grows at most polynomially in time
\begin{align*}
\|DD\phi\|_{L^2}\les 1+ (1+t)^{\frac{2+\epsilon}{2+\epsilon-2s(1+\epsilon) }}.
\end{align*}

Finally to improve the pointwise decay estimate for the solution for larger power of $p$, we rely on the conformal symmetry of the system 
\[
\phi \mapsto (1+t+|x|)^{\frac{1}{2}}(1+t-|x|)^{\frac{1}{2}} \phi, \quad A\mapsto A 
\]
and the proof for the main theorem \ref{thm:main:pd:imp} is reduced to showing that solution to the system 
\begin{equation*}
 \Box_{A} \phi =\Lambda^{\frac{5-p}{2}} | \phi|^{p-1} \phi, \quad  F_{\mu\nu}=\epsilon_{\mu\nu\ga} J[ \phi]^{\ga},\quad \Lambda=(1-t+|x|)(1-t-|x|)
 \end{equation*} 
 on the backward cone $B=\{t+|(x_1, x_2)|\leq 1\}$ with compactly supported initial data is uniformly bounded up to the tip point. However to better make use of the geometric Kirchoff-Sobolev paramatrix, we lift the above equation to space dimension three by simply assuming that $A$ and $\phi$ is independent of the third variable $x_3$. The cone $B$ will be changed to $\{t+|(x_1, x_2, x_3)|\leq 1\} $ accordingly. For any fixed point $q=(t_0, x_0)$, the solution $\phi$ at $q$ can be represented as follows
\begin{equation*}
\begin{split}
4\pi|\phi(q) | &=I.D. +\iint_{\mathcal{N}^{-}(q)}\langle\widetilde{\lap}_A h-i\widetilde{\rho}  h, \phi\rangle\widetilde{r} d\widetilde{r}d\widetilde{\om}  -\iint_{\mathcal{N}^{-}(q)}\langle  h, \Box_A\phi\rangle\widetilde{r} d\widetilde{r}d\widetilde{\om},
\end{split}
\end{equation*}
where  $h$ is  the function on the backward light cone $\mathcal{N}^{-}(q)=\{t-t_0+|x-x_0|=0\}$ such that 
\begin{align*}
D_{\tilde{\Lb}} h=0, \quad h(q)= |\phi(q)|^{-1}\phi(q).
\end{align*}
The tilde components are those associated to the coordinates centered at the point $q$. 
The first term relies only on the initial data which is clearly uniformly bounded. The third term is the pure power nonlinearity which could be controlled in a similar way as the flat case in \cite{yang:NLW:2D}. The most difficult term is the second one arising from the connection field $A$. 
Like the Yang-Mills-Higgs system studied in \cite{YangYu:MKGrough}, commuting the transport equation for $h$ with angular derivatives, one first can obtain the bound
\begin{align*}
|\tilde{r}^2(\lap_A h-i\tilde{\rho} h)| 
&\leq |\int_0^{\tilde{r}}  \tilde{s}^2 \pa^\mu F(\tilde{\Lb}, \pa_{\mu}) ds|+  \big( \int_0^{\tilde{r}} F(\tilde{\Lb},  \tilde{\Omega}_{jk}) ds\big)^2.
\end{align*}
Now the first term is easier since 
\begin{align*}
|\pa^\mu F(\tilde{\Lb}, \pa_\mu )|\les  |\phi|^6+ |D_{\tilde{\Lb}}\phi|^2+|\tilde{\D}\phi|^2,
\end{align*}
which after integration over the unit sphere can be bounded by using the standard energy conservation (recall that we have lifted the problem to space dimension three) through the backward light cone $\mathcal{N}^{-}(q)$. The second term involves the Maxwell field  $F$, for which the energy estimate is absent for solution to the Chern-Simons-Higgs equation. This is in vast  contrast to the Maxwell-Klein-Gordon system \cite{YangYu:MKG:smooth} for which the energy  controls both the scalar field and the Maxwell field. Now we first  bound that 
\begin{align*}
\big( \int_0^{\tilde{r}} F(\tilde{\Lb},  \tilde{\Omega}_{jk}) ds\big)^2\les  \big( \int_0^{\tilde{r}}    \tilde{s} \Im(\phi\cdot \overline{D_{\tilde{\Lb}}\phi})  ds\big)^2   + \big(\int_0^{\tilde{r}} |\tilde{x}_3| (|D_{\tilde{\Lb}}\phi|+|\tilde{\D}\phi|)|\phi| ds \big)^2. 
\end{align*} 
Although the rough bound for the solution $\phi$ may blow up on the boundary, the key observation is that it can only go to infinity at $\tilde{x}_3=0$. This structure allows us to control the second term by using the rough pointwise estimate  for $\phi$ obtained in Thoerem \ref{thm:main:pd:gen}. 

It then remains to control the first term for which we make use of the geometric bilinear trace theorem developed by Klainerman-Rodnianski in \cite{Klainerman06:trace} (see Theorem \ref{them:Gbilinear} in the last section). Roughly speaking, the backward light cone $\mathcal{N}^{-}(q)$ is null hypersurface which can be written as $[0, \tilde{r}]\times \mathbb{S}^2$. Note that $\tilde{\Lb}$ is the tangential derivative. Then the bilinear trace theorem states that 
\[
 \int_{\mathbb{S}^2}\left(\int_0^{1} f\cdot D_{\tilde{\Lb}} g ds \right)^2 d\tilde{\omega} \les \mathcal{N}_1(f) \mathcal{N}_1(g),
\]
in which  $ \mathcal{N}_1(f)=\int_0^1 \int_{\mathbb{S}^2} |D_{\tilde{\Lb}}f|^2+|\D f|^2 dsd\omega$. This is closely related to the energy flux through the null hypersurface. Then  a simple scaling argument could lead to the desired bound.

\bigskip

The paper is organized as follows: In section 2, we apply the new vector field introduced by the authors in \cite{yang:NLW:2D} to obtain time decay for the potential energy. Then in section 3, we show rough pointwise decay estimate for the solution. The main new difficulty is to establish the polynomial growth of the second order energy, which we rely on Strichartz estimate. The last section is devoted to the improvement of pointwise decay estimate for the solution for larger power of $p$. This is realized by using conformal transformation and geometric Kirchoff-Sobolev paramatrix for geometric wave equations.  

\textbf{Acknowledgments.} D. Wei is partially supported by the National Key R\&D Program of China  2021YFA1001500. S. Yang is supported by the National Key R\&D Program of China 2021YFA1001700 and the National Science Foundation of China 12171011,  12141102.


\section{Potential energy decay}
As in the flat case when the connection field $A$ vanishes, the starting point of the decay mechanism of the solution to this nonlinear system is the time decay of the potential energy. 
\begin{Prop}
  \label{prop:td:2dCSH}
  For solution $\phi$ to the nonlinear equation \eqref{eq:CSH:p:2d} in $\mathbb{R}^{1+2}$, we have the time decay of the  potential energy
  \begin{align*}
    \int_{\mathbb{R}^2}|\phi(t, x)|^{p+1}(1+t+|x|)^{\min\{\frac{p-1}{2}, 2\}}dx\leq C\mathcal{E}_{0, 2}
    ,\quad \forall t\geq 0
  \end{align*}
  for some constant $C$ depending only on $p$.

\begin{remark}
Same time decay for the superconformal case when $p\geq 5$ and a weaker decay estimate for the case when $3\leq p \leq 5$ had been shown by Huh in \cite{Huh17:CS:conformal}. The method relies on the conformal energy identity used for the flat case by Pecher in \cite{Pecher82:NLW:2d}. 
\end{remark}

\end{Prop}
The rest of this section is devoted to the proof of this proposition, which is similar to that in \cite{yang:NLW:2D}. Without loss of generality, we only study the solution in the future, that is, $t\geq 0$. For simplicity, in this section $A\les B$ stands for  $A\leq CB$ for some constant $C$ depending only on $p$.

The proof relies on the vector field method originally introduced by Klainerman in \cite{klinvar}, which 
turns out to be quite robust for treating such geometric nonlinear wave equations. 
Define the energy momentum tensor for the scalar field $\phi$
\begin{align*}
  T[\phi]_{\mu\nu}=\langle D_{\mu}\phi, D_{\nu}\phi \rangle-\f12 m_{\mu\nu}(\l D^\ga \phi,  D_\ga\phi\r+\frac{2}{p+1} |\phi|^{p+1}),
\end{align*}
where $\l a, b \r=\f12(a\bar{b}+\bar{a}b)$ means  the inner product for complex numbers $a$ and $b$. We can compute that 
\begin{align*}
\pa^\mu T[\phi]_{\mu\nu} = \l \Box_A\phi-|\phi|^{p-1}\phi,  D_\nu \phi \r  + \l D^\ga \phi, iF_{\ga\nu} \phi\r=\l \Box_A\phi-|\phi|^{p-1}\phi,  D_\nu \phi\r  - F_{\ga\nu} J^\ga[\phi].
\end{align*}
Here recall that $J_{\mu}[\phi]=\Im(\phi \cdot \overline{D_\mu\phi})$ is the current defined in the equation \eqref{eq:CSH:p:2d}. 
For any vector field $X$ and any function $\chi$ defined on $\mathbb{R}^{1+2}$, denote  the current
\begin{equation*}
P^{X,  \chi}_\mu[\phi]=T[\phi]_{\mu\nu}X^\nu -
\f12\pa_{\mu}\chi \cdot|\phi|^2 + \f12 \chi\pa_{\mu}|\phi|^2.
\end{equation*}
We can compute that 
\begin{align*}
\pa^\mu P^{X,  \chi}_\mu[\phi]= &T[\phi]^{\mu\nu}\pi^X_{\mu\nu}+\l \Box_A\phi-|\phi|^{p-1}\phi,  D_X \phi\r  - X^\nu F_{\ga\nu} J^\ga[\phi]-\f12 \Box\chi \cdot|\phi|^2\\
&+ \chi (\l \phi, \Box_A\phi\r+\l D_\mu\phi , D^\mu\phi \r ),
\end{align*}
in which $(\pi^X)^{\mu\nu}=\f12 (\pa^\mu X^\nu+\pa^\nu X^\mu )$  is the deformation tensor of the metric $m$ along the vector field $X$.
For any domain $\mathcal{D}$ in $\mathbb{R}^{1+2}$, by using Stokes' formula, we have the energy identity
\begin{align}
\label{eq:energy:id}
\int_{\pa\mathcal{D}}i_{  P^{X, \chi}[\phi]} d\vol  =\iint_{\mathcal{D}}  &T[\phi]^{\mu\nu}\pi^X_{\mu\nu}+\l \Box_A\phi-|\phi|^{p-1}\phi,  D_X \phi \r - X^\nu F_{\ga\nu} J^\ga[\phi] \\
\notag
&-\f12\Box\chi\cdot|\phi|^2+ \chi (\l \phi , \Box_A\phi\r+\l D_\mu\phi , D^\mu\phi \r )  d\vol,
\end{align}
where $i_Z d\vol$ is the contraction of the vector field $Z$ with the volume form $d\vol$ and $\Box$ is the wave operator corresponding to the flat metric $m$ in $\mathbb{R}^{1+2}$.

For solution to the Chern-Simons-Higgs equation \eqref{eq:CSH:p:2d}, we first note that  
\begin{align*}
&\l \Box_A\phi-|\phi|^{p-1}\phi,   D_X \phi \r  - X^\nu F_{\ga\nu} J^\ga[\phi]+ \chi (\l \phi, \Box_A\phi\r+\l D_\mu\phi,  D^\mu\phi\r  )\\
&=-X^\nu \varepsilon_{\ga\nu\mu}J^{\mu}[\phi] J^{\ga}[\phi]+\chi (|\phi|^{p+1} +\l D_\mu\phi,  D^\mu\phi\r) \\
&=\chi (|\phi|^{p+1} +\l D_\mu\phi ,  D^\mu\phi\r )
\end{align*}
as $\varepsilon_{\ga \nu \mu}$ is skew symmetric.

In space dimension two for the case when $1<p\leq 5$, instead of taking regions bounded by null cones, consider the region 
\begin{align*}
\mathcal{D}=\{(t, x_1, x_2)| 0\leq t\leq \min\{x_1, T\}\}
\end{align*}
bounded by the hyperplane $\{t=x_1\}$, the initial hypersurface $\{t=0\}$ and the constant $t$-slice $\{t=T\}$ for any positive constant $T$. In the exterior region, first take the vector field $X=\partial_t$ and the function $\chi=0$. Since $\pa_t$ is Killing, the deformation tensor vanishes $\pi^X=0$. In particular, the bulk term on the right hand side of the energy identity \eqref{eq:energy:id} vanishes and hence we derive the classical energy conservation. To compute the boundary integrals, define the null coordinates and frame
$$u_1=t-x_1+1, \quad v_1=t+x_1, \quad  \L_1=2\pa_{v_1}=\partial_t+\partial_1,\quad \Lu_1=\pa_t-\pa_{1}.$$
Recall the volume form
\[
d\vol=dtdx=dtdx_1dx_2= \f12 du_1 dv_1 dx_2. 
\]
On the null hyperplane $\{t-x_1=0\}$, we have
\begin{align*}
i_{P^{X, \chi}[\phi]}d\vol&=\f12 (P^{X, \chi}[\phi])^{u_1}dv_1 dx_2=   -(P^{X, \chi}[\phi])_{v_1} 2dt dx_2\\
&=-(\l D_{t}\phi, (D_{t}+D_{1})\phi\r-\f12(-|D_t\phi|^2+|D_1\phi|^2+|D_2\phi|^2+\frac{2}{p+1}|\phi|^{p+1}) )dtdx_2\\
&=-\f12(|D_{\L_1}\phi|^2+|D_2\phi|^2+\frac{2}{p+1}|\phi|^{p+1}) dtdx_2.
\end{align*}
On the constant $t$-slice, we compute that
\begin{align*}
  i_{P^{X,\chi}[\phi]}d\vol=&(P^{X, \chi}[\phi])^{0}dx= - T[\phi]_{0 0}  dx\\
=&-\f12(|D_t\phi|^2+|D_1\phi|^2+|D_2\phi|^2+\frac{2}{p+1}|\phi|^{p+1})   dx.
\end{align*}
We hence derive the classical energy conservation in the exterior region
\begin{align*}
  &\int_{0\leq t\leq T}(|D_{\L_1}\phi|^2+|D_2\phi|^2+\frac{2}{p+1}|\phi|^{p+1})(t,t,x_2)dtdx_2  +\int_{x_1\geq T}  (|D \phi|^2 +\frac{2}{p+1}|\phi|^{p+1})(T, x) dx \\
  &=\int_{x_1\geq 0} (  |D\phi|^2 +\frac{2}{p+1}|\phi|^{p+1})(0, x) dx.
\end{align*}
Here $D\phi$ is short for $(D_t\phi, D_1\phi, D_2\phi)$. 
By letting $T\to+\infty$, it then holds that 
\begin{align}
\label{eq:E2}
  \int_{t\geq 0}(|D_{\L_1}\phi|^2+|D_2\phi|^2+\frac{2}{p+1}|\phi|^{p+1})(t,t,x_2)dtdx_2 \leq   \mathcal{E}_{0, 0}.
\end{align}
Similarly take $X=\partial_t$, $\chi=0$ and $\mathcal{D}=[0,t]\times\R^2$ in the energy identity \eqref{eq:energy:id} we have
\begin{align}
\label{eq:E3}
  \int_{\R^2}(|D\phi|^2+\frac{2}{p+1}|\phi|^{p+1})(t,x)dx=\int_{\R^2}(|D\phi|^2+\frac{2}{p+1}|\phi|^{p+1})(0,x)dx \leq   \mathcal{E}_{0, 0}.
\end{align}
Next we choose the vector field $X$ and the function $\chi$ to be  
\begin{align*}
X= (x_2^2+(t-x_1)^2)\partial_t+(x_2^2-(t-x_1)^2)\partial_1+2(t-x_1)x_2\partial_2, \quad  \chi=t-x_1.
\end{align*}
We can compute the deformation tensor
\[
\pi^X_{\mu\nu}= 2(t-x_1)m_{\mu\nu}
\]
and it is obvious that 
\[
\partial_t\chi=1, \quad (\partial_t+\partial_1)\chi=0,\quad \Box\chi=0.
\]
Here recall that $m_{\mu\nu}$ is the flat  Minkowski metric on $\mathbb{R}^{1+2}$. For such choice of vector field,  the bulk term on the right hand side of the energy identity \eqref{eq:energy:id} becomes 
\begin{align*}
& T[\phi]^{\mu\nu}\pi^X_{\mu\nu}+
\chi ( \l D_\mu\phi, D^\mu\phi \r +|\phi|^{p+1}) \\
&=2\chi \left(\l D_\ga \phi, D^\ga \phi\r -\frac{3}{2}(\l D^\ga \phi , D_\ga \phi\r +\frac{2}{p+1}|\phi |^{p+1})\right) +\chi ( \l D_\mu\phi, D^\mu\phi \r +|\phi|^{p+1})\\
&=-\frac{6(t-x_1)}{p+1}|\phi|^{p+1} +(t-x_1) |\phi|^{p+1}= -\frac{(5-p)(t-x_1)}{p+1}|\phi|^{p+1}.
\end{align*}
This implies that the bulk integral is nonnegative restricted to the exterior region when $ t\leq x_1$ for the subconformal case $p\leq 5$.

For the boundary integrals on the left hand side of the energy identity \eqref{eq:energy:id} with the above chosen vector field $X$ and function $\chi$, similarly on the null hyperplane $\{t-x_1=0\}$, we have
\begin{align*}
i_{P^{X, \chi}[\phi]}d\vol&=-(P^{X, \chi}[\phi])_{\L_1} dt dx_2\\
&= -( T[\phi]_{\L_1 \nu}X^\nu -
\f12 \L_1\chi |\phi|^2 + \f12 \chi\cdot \L_1 |\phi|^2)  dtdx_2\\
&=-T[\phi]_{\L_1 X} dtdx_2=-x_2^2T[\phi]_{\L_1\L_1} dtdx_2=-x_2^2|D_{\L_1}\phi|^2dtdx_2.
\end{align*}
 Here we note that on the null hyperplane  $\{t-x_1=0\}$ it holds that 
\[
\L_1\chi=0,\quad \chi=0,\quad X=x_2^2\L_1.
\]
Similarly on the constant $t$-slice $\{t=0\}$ or $\{t=T\}$, we have
\begin{align*}
  i_{P^{X,\chi}[\phi]}d\vol=&(P^{X, \chi}[\phi])^{0}dx= -( T[\phi]_{0 \nu}X^\nu -
\f12 \pa_t\chi |\phi|^2 + \f12 \chi\cdot \pa_t |\phi|^2)  dx\\
=&-\f12\big( (x_2^2+(t-x_1)^2) (|D\phi|^2+\frac{2}{p+1}|\phi|^{p+1})- |\phi|^2 +(t-x_1)   \pa_t |\phi|^2  \\
&\quad \qquad+2\l D_t  \phi, (x_2^2-(t-x_1)^2)D_1\phi+2(t-x_1)x_2 D_2\phi \r \big)  dx\\
=&-\f12( |x_2 D_{\L_1}\phi+(t-x_1) D_2\phi|^2+|(t-x_1) (D_t\phi-D_1\phi)+x_2D_2\phi+\phi|^2\\
&\qquad +\frac{2(x_2^2+(t-x_1)^2)}{p+1}|\phi|^{p+1}  +\partial_1((t-x_1)|\phi|^2)-\partial_2(x_2|\phi|^2))  dx.
\end{align*}
Since the initial data decay sufficiently fast, integration by parts on the hypersurface $$\{(t, x_1, x_2)| t=t, x_1\geq 0\}$$ leads to 
\begin{align*}
  \int_{ x_1\geq t} (\partial_1((x_1-t)|\phi|^2)+\partial_2(x_2|\phi|^2))(t,x)dx=0.
\end{align*}
Combining the above computations, in view of the energy identity \eqref{eq:energy:id}, we obtain the following energy identity
\begin{align*}
  &2\iint_{0\leq t\leq \min(x_1,T)}  \frac{(5-p)(x_1-t)}{p+1}|\phi|^{p+1}d\vol +2\int_{0\leq t\leq T}x_2^2|D_{\L_1}\phi|^2(t,t,x_2)dtdx_2 \\
  &+ \int_{x_1\geq T}  |x_2 D_{L_1}\phi-(T-x_1)D_2\phi|^2+|(T-x_1)D_{\Lu_1}\phi+D_2(x_2\phi)|^2
  +\frac{2(x_2^2+(T-x_1)^2)}{p+1}|\phi|^{p+1}  dx \\
  &= \int_{x_1\geq t}  (|x_2 D_{L_1}\phi-x_1D_2\phi|^2+|-x_1 D_{\Lu_1}\phi+D_2(x_2\phi)|^2
  +\frac{2(x_2^2-x_1^2)}{p+1}|\phi|^{p+1} ) dx,
\end{align*}
from which we obtain the weighted energy flux bound 
\begin{align}
\label{eq:E1}
  &\int_{t\geq 0}x_2^2|D_{\L_1}\phi|^2(t,t,x_2)dtdx_2\\
  \notag
  &\leq \f12\int_{x_1\geq 0}  |x_2(\phi_1+D_1\phi_0)-x_1D_2\phi_0|^2  +|x_1(D_1\phi_0-\phi_1)+x_2D_2\phi_0+\phi_0|^2 
+\frac{2|x|^2}{p+1}|\phi_0|^{p+1}  dx\\
\notag
&\les  \int_{\{t=0\}}(1+|x|^2)(|\phi_1|^2+|\bar{D}\phi_0|^2+ |\phi_0|^{p+1})dx = \mathcal{E}_{0, 2}.
\end{align}
Here recall that the implicit constant  relies only on $p$ and the integral of $|\phi_0|^2$ can be controlled by using a type of  Hardy's inequality.

\bigskip

Finally in the interior region $\{t\geq x_1\}$, for the vector field $X$ and the function $\chi$
\begin{align*}
X&=u_1^q(\partial_t-\partial_1)+u_1^{q-2}x_2^2(\partial_t+\partial_1)+2u_1^{q-1}x_2\partial_2, \\
 \chi&=u_1^{q-1},\quad u_1=t-x_1+1,\quad q=\f12(p-1),
\end{align*}
we can compute the covariant derivatives
\begin{align*}
\nabla_{\L_1}X= 0,\quad &\nabla_{\Lu_1}X= 2qu_1^{q-1}\Lu_1+2(q-2)u_1^{q-3}x_2^2\L_1+4(q-1)u_1^{q-2}x_2\partial_2,\\
&\nabla_{2}X=2u_1^{q-2}x_2\L_1+2u_1^{q-1}\partial_2
\end{align*}
in terms of the null frame $\{\L_1, \Lu_1, \pa_2\}$. 
Therefore we have the nonvanishing components of the deformation tensor $\pi_{\mu\nu}^X$
\begin{align*}
\pi^X_{L_1\Lb_1}&=-2qu_1^{q-1},\quad \pi^X_{\Lb_1\Lb_1}=-4(q-2)u_1^{q-3}x_2^2,\\
 \pi^X_{\Lb_1 \pa_2}&=2(q-2)u_1^{q-2}x_2,\quad \pi^X_{\pa_2\pa_2}=2u_1^{q-1}.
\end{align*}
Note $\Box\chi=\Box u_1^{q-1} =0$. The bulk term on the right hand side of the energy identity \eqref{eq:energy:id} is of the form 
\begin{align*}
& T[\phi]^{\mu\nu}\pi^X_{\mu\nu}+
\chi (\l D_\mu\phi, D^\mu\phi \r + |\phi|^{p+1})-\f12 \Box\chi |\phi|^2\\
&=-qu_1^{q-1} (|D_2\phi|^2+\frac{2}{p+1}|\phi|^{p+1})+ 2u_1^{q-1}(|D_2\phi|^2-\frac{1}{2}\l D^\mu \phi , D_\mu\phi\r-\frac{1}{p+1}|\phi|^{p+1})\\
&\quad -(q-2)u_1^{q-3}x_2^2 |D_{\L_1}\phi|^2-2(q-2)u_1^{q-2}\l x_2 D_2\phi, D_{\L_1}\phi\r+
u_1^{q-1} \l D_\mu\phi, D^\mu\phi\r  +u_1^{q-1} |\phi|^{p+1} \\
&=(2-q)u_1^{q-3} |x_2D_{\L_1}\phi+u_1 D_2\phi|^2 .
\end{align*}
This shows that for the subconformal case when $p\leq 5$, the bulk term is  nonnegative as $q=\frac{p-1}{2}\leq 2$. 

Apply the energy identity \eqref{eq:energy:id} to the above chosen vector field $X$ and function $\chi$ for the region  $\mathcal{D}=\{\max(x_1,0)\leq t\leq T\}$. For the boundary integrals on the left hand side, on the null hyperplane $\{t-x_1=0\}$,
note that 
\[
\chi= u_1=1, \quad X=\Lu_1+x_2^2\L_1+2x_2\partial_2.
\]
Therefore   we can compute that 
\begin{align*}
i_{P^{X, \chi}[\phi]}d\vol&=-(P^{X, \chi}[\phi])_{L_1}dtdx_2= -( T[\phi]_{\L_1 \nu}X^\nu -
\f12 \L_1\chi |\phi|^2 + \f12 \chi\cdot \L_1 |\phi|^2)  dtdx_2\\
&=-(x_2^2T[\phi]_{\L_1\L_1}+T[\phi]_{\L_1\Lu_1}+2x_2T[\phi]_{\L_1 \pa_2}+ \f12 \L_1 |\phi|^2) dtdx_2\\
&=-(x_2^2|D_{\L_1}\phi|^2+|D_2\phi|^2+\frac{2}{p+1}|\phi|^{p+1}+2\l x_2D_2\phi, D_{\L_1}\phi\r+ \f12 \L_1 |\phi|^2) dtdx_2\\
&=-(|x_2D_{\L_1}\phi+D_2\phi|^2+\frac{2}{p+1}|\phi|^{p+1}+ \f12 \L_1 |\phi|^2) dtdx_2.
\end{align*}
On the part of the boundary $\pa\cD\cap \{t-x_1=0\}$, integration by parts again, we see that 
\begin{align*}
-\f12\int_{\pa\cD\cap \{t-x_1=0\}} \L_1 |\phi|^2 dtdx_2&= -\frac{1}{2}\int_{\pa\cD\cap \{t-x_1=0\}} \pa_{v_1} |\phi|^2 dv_1 dx_2 \\
&=-\f12 (\int_{t=T=x_1}|\phi(T, T, x_2)|^2 dx_2- \int_{t=0=x_1}|\phi(0, 0, x_2)|^2 dx_2).
\end{align*}
Similarly  on the constant $t$-slice, we  have 
\begin{align*}
  i_{P^{X,\chi}[\phi]}d\vol=&(P^{X, \chi}[\phi])^{0}dx= -( T[\phi]_{0 \nu}X^\nu -
\f12 \pa_t\chi |\phi|^2 + \f12 \chi\cdot \pa_t |\phi|^2)  dx\\
=&-\f12u_1^{q-2}\Big( (x_2^2+u_1^2) (|D\phi|^2+\frac{2}{p+1}|\phi|^{p+1})- (q-1)|\phi|^2  +u_1   \pa_t |\phi|^2 \\
&\qquad\qquad  +2 (x_2^2-u_1^2)\l D_t\phi, D_1\phi\r+4u_1 x_2 \l D_t\phi , D_2\phi\r \Big)  dx\\
=&-\f12 u_1^{q-2}\Big( |x_2D_{\L_1}\phi+u_1 D_2\phi|^2+|u_1 D_{\Lu_1}\phi+D_2(x_2\phi)|^2 +\frac{2(x_2^2+u_1^2)}{p+1}|\phi|^{p+1} \Big)dx\\
&- \f12 \Big( \partial_1(u_1^{q-1}|\phi|^2)- \partial_2(x_2u_1^{q-2}|\phi|^2)\Big)  dx.
\end{align*}
On the part of the boundary  $\pa\cD\cap \{t=\tau\}$ ($\tau=0$ or $T$), integration by parts again, we can show that 
\begin{align*}
 - \f12 \int_{\pa\cD\cap \{t=\tau\}}\Big( \partial_1(u_1^{q-1}|\phi|^2)- \partial_2(x_2u_1^{q-2}|\phi|^2)\Big)  dx=- \f12 \int_{\pa\cD\cap \{t=\tau\}}   |\phi|^2(\tau, \tau, x_2)   dx_2.
\end{align*}
Here we always use the fact that the solution decays sufficiently fast at spatial infinity (one could also take compact regions to approximate the region $\cD$). These computations show that 
\begin{align*}
  \int_{x_1\leq t}(\partial_1(u_1^{q-1}\phi^2)-\partial_2(x_2u_1^{q-2}\phi^2))dx\Big|_{t=0}^{t=T}=\int_{0\leq t\leq T} (\L_1 |\phi|^2)(t,t,x_2)dtdx_2.
\end{align*}
In other words, those terms on the boundary without a definite sign are canceled. 
Therefore combining the above computations, we obtain the energy identity
\begin{align*}
  &\int_{x_1\leq T} u_1^{q-2}(  |x_2D_{\L_1}\phi+u_1 D_2\phi|^2+|u_1 D_{\Lu_1}\phi+D_2(x_2\phi)|^2 +\frac{2(x_2^2+u_1^2)}{p+1}|\phi|^{p+1})(T, x) dx   \\
  &+2\iint_{\cD}  (2-q)u_1^{q-3} |x_2D_{\L_1}\phi+u_1D_2\phi|^2d\vol  \\ 
  &=2\int_{0\leq t\leq T}(|x_2D_{\L_1}\phi+D_2\phi|^2+\frac{2}{p+1}|\phi|^{p+1})(t,t,x_2)dtdx_2\\
  &+\int_{x_1\leq 0} u_1^{q-2}(  |x_2D_{\L_1}\phi+u_1 D_2\phi|^2+|u_1 D_{\Lu_1}\phi+D_2(x_2\phi)|^2 +\frac{2(x_2^2+u_1^2)}{p+1}|\phi|^{p+1})(0,x)dx.
\end{align*}
Recall that on the initial hypersurface $\{t=0\}$ it holds that 
\begin{align*}
\int_{x_1\leq 0} u_1^{q-2}(  |x_2D_{\L_1}\phi+u_1 D_2\phi|^2+|u_1 D_{\Lu_1}\phi+D_2(x_2\phi)|^2 +\frac{2(x_2^2+u_1^2)}{p+1}|\phi|^{p+1})(0,x)dx
 \les \mathcal{E}_{0, 2}. 
\end{align*}
In view of the previous estimates \eqref{eq:E1} and \eqref{eq:E2}, we can bound the weighted energy flux through the null hyperplane
\begin{align*}
  & \int_{0\leq t\leq T}(|x_2D_{\L_1}\phi+D_2 \phi|^2+\frac{2}{p+1}|\phi|^{p+1})(t,t,x_2)dtdx_2\\
  & \leq 2\int_{0\leq t\leq T}(x_2^2|D_{\L_1}\phi|^2+|D_2\phi|^2+\frac{2}{p+1}|\phi|^{p+1})(t,t,x_2)dtdx_2 
  \les \mathcal{E}_{0, 2}.
\end{align*}
Since $q=\frac{1}{2}(p-1)\leq 2$, we therefore obtain the weighted potential energy 
estimate 
\begin{align*}
  \int_{x_1\leq T}\frac{2(u_1^{q}+u_1^{q-2}x_2^2)|\phi|^{p+1}}{p+1}dx\les  \mathcal{E}_{0, 2},\quad \forall T\geq 0.
\end{align*}
Recall that $u_1=t-x_1+1$. Restricting the above integral to the region $\{x_1\leq 0\}$, we obtain that 
\begin{align*}
  \int_{x_1\leq 0} (t+1)^q |\phi|^{p+1}(t,x)dx &\leq \int_{x_1\leq t} u_1^q |\phi|^{p+1}(t,x)dx \les \mathcal{E}_{0, 2},\\
  \int_{x_1\leq t\leq r} (1+2r)^{q-2}x_2^2|\phi|^{p+1} dx&\leq \int_{x_1\leq t}  u_1^{q-2}x_2^2|\phi|^{p+1} dx \les \mathcal{E}_{0, 2}.
\end{align*}
Here $r=|x|=\sqrt{x_1^2+x_2^2}$.
This in particular shows that
\begin{align*}
  \int_{x_1\leq 0} (1+t)^{\frac{p-1}{2}} |\phi|^{p+1}(t,x)dx+\int_{x_1\leq 0, r\geq t} (1+r)^{\frac{p-1}{2}-2} x_2^2 |\phi|^{p+1}(t,x)dx\les \mathcal{E}_{0, 2}.
\end{align*}
By symmetry (or changing variable $x_1\rightarrow -x_1$), we also have
\begin{align*}
   \int_{x_1\geq 0} (1+t)^{\frac{p-1}{2}} |\phi|^{p+1}(t,x)dx+\int_{x_1\geq 0, r\geq t} (1+r)^{\frac{p-1}{2}-2} x_2^2 |\phi|^{p+1}(t,x)dx\les \mathcal{E}_{0, 2},
\end{align*}
which together with the previous estimate implies that 
\begin{align*}
   \int_{\mathbb{R}^2} (1+t)^{\frac{p-1}{2}} |\phi|^{p+1}(t,x)dx+\int_{r\geq t} (1+r)^{\frac{p-1}{2}-2} x_2^2 |\phi|^{p+1}(t,x)dx\les \mathcal{E}_{0, 2}.
\end{align*}
Now by symmetry again with $x_1$ to $x_2$, we also have
\begin{align*}
   \int_{r\geq t} (1+r)^{\frac{p-1}{2}-2} x_1^2 |\phi|^{p+1}(t,x)dx\les \mathcal{E}_{0, 2}.
\end{align*}
These estimates are sufficient to show the weighted potential energy estimate
\begin{align*}
  \int_{\R^2}(1+t+r)^{\frac{p-1}{2}}|\phi|^{p+1}(t,x)dx\les \mathcal{E}_{0, 2}.
\end{align*}
In particular the proposition holds for the subconformal case when $p\leq 5$.

\bigskip

For the superconformal case when $p\geq 5$, we can directly use the conformal Killing vector field 
\begin{align*}
X=(t^2+r^2)\pa_t+2tr \pa_r,\quad \chi=t.
\end{align*}
We can compute the deformation tensor 
\begin{align*}
\pi^X_{\mu\nu}=2tm_{\mu\nu}.
\end{align*}
Thus for the bulk integral, we can demonstrate that 
\begin{align*}
& T[\phi]^{\mu\nu}\pi^X_{\mu\nu}+
\chi (\l D_\mu\phi, D^\mu\phi \r + |\phi|^{p+1})-\f12 \Box\chi |\phi|^2\\
&=  2t \Big(\l D^\mu \phi, D_\mu\phi \r-\frac{3}{2}(\l D_\mu\phi, D^\mu\phi \r+\frac{2}{p+1}|\phi |^{p+1})\Big)+t (\l D_\mu\phi, D^\mu\phi \r + |\phi|^{p+1})\\
&=\frac{(p-5)t}{p+1}|\phi|^{p+1}.
\end{align*}
In particular the bulk integral on the right hand side of the energy identity \eqref{eq:energy:id} is nonnegative for the superconformal case $p\geq 5$ with the above conformal Killing vector field. 

Next for the boundary integrals for the region $\cD$ bounded by the constant $t$-slice $\{t=T\}$ and the initial hypersurface $\{t=0\}$, we first have 
\begin{align*}
  & i_{P^{X,\chi}[\phi]}d\vol=(P^{X, \chi}[\phi])^{0}dx= -( T[\phi]_{0 \nu}X^\nu -
\f12 \pa_t\chi |\phi|^2 + \f12 \chi\cdot \pa_t |\phi|^2)  dx\\
=&- \f12 \Big(  (t^2+r^2) (|D\phi|^2+\frac{2}{p+1}|\phi|^{p+1})-  |\phi|^2  +t   \pa_t |\phi|^2   +4tx_1 \l D_t\phi, D_1\phi\r+4t x_2 \l D_t\phi, D_2\phi\r \Big)  dx\\
=& - \f12 \Big(\frac{2(t^2+r^2)}{p+1}|\phi|^{p+1}+ |tD_t\phi +x_1 D_1\phi +x_2 D_2\phi +\phi|^2+|tD_1\phi +x_1 D_t\phi|^2\\
&\qquad +|tD_2\phi +x_2 D_t\phi |^2+|x_2 D_1\phi -x_1 D_2\phi |^2-\pa_1(x_1|\phi|^2)-\pa_2(x_2|\phi |^2)   \Big)  dx.
\end{align*}
On the constant $t$-slice $\{t=\tau\}$, note that 
\begin{align*}
\int_{t=\tau}  \pa_1(x_1|\phi|^2)+\pa_2(x_2|\phi |^2)  dx=0
\end{align*}
as long as the solution decays suitably fast at spatial infinity. We then obtain the conformal energy identity
\begin{align}
\label{eq:Eid:conf}
\int_0^T\int_{\mathbb{R}^2} \frac{(p-5)t}{p+1}|\phi|^{p+1} d\vol +\int_{t=T}  \frac{t^2+r^2}{p+1}|\phi|^{p+1} +Q[\phi] dx = \int_{t=0}  \frac{t^2+r^2}{p+1}|\phi|^{p+1} +Q[\phi] dx,
\end{align}
in which 
\begin{align}
\label{eq:def4Q}
Q[\phi]=\f12 \Big( |D_S\phi +\phi|^2+|tD_1\phi +x_1 D_t\phi|^2  +|tD_2\phi +x_2 D_t\phi |^2+|D_{\Omega}\phi|^2  \Big).
\end{align}
Here $S=t\pa_t+r\pa_r$ is the scaling vector field and $\Omega=x_1\pa_2-x_2\pa_1$ is the angular momentum vector field. 
Since $p\geq 5$ and $Q[\phi]\geq 0$, in view of the standard energy conservation  \eqref{eq:E3}, we in particular derive that 
\begin{align*}
\int_{\mathbb{R}^2}(t^2+r^2+1)|\phi|^{p+1}(t, x)dx\les  \int_{\mathbb{R}^2}  ( 1+r^2)|\phi|^{p+1}(0, x)  +Q[\phi](0, x)  dx\les \mathcal{E}_{0, 2}.
\end{align*}
We thus finished the proof for Proposition \ref{prop:td:2dCSH}.

\section{Rough pointwise decay estimate for all $p$}
As we have pointed previously, the decay mechanism of the solution is the potential energy decay obtained in the previous section. This section is devoted to showing the pointwise decay estimates for the solution by using a type of Br\'{e}zis-Gallouet-Wainger (see for example in  \cite{Brezis80:LSobolev}, \cite{Brezis80:LSobolev:g}) inequality. Although the solution decays faster for larger $p$ as discussed in the introduction, in view of the potential energy decay in Proposition \ref{prop:td:2dCSH}, the solution decays at most like linear wave for the superconformal case when $p\geq 5$. In other words, the case when $p\geq 5$ is similar to that for the case when $p=5$. For simplicity, we only discuss the situation for the subconformal case $p\leq 5$. In this section we  derive a rough decay estimate for the solution for all $1<p\leq 5$.

One difficulty in space dimension two is the failure of the classical Sobolev embedding for the $L^\infty$ bound by the standard energy. This can be  compensated  by using a type of logarithmic Sobolev inequality. In $\R^2$, let 
\[ 
B_R=\{x:|x|\leq R\},\quad B_R^c=\{x: |x|> R\}.
 \]
 We first recall the Br\'{e}zis-Gallouet-Wainger logarithmic Sobolev inequality.
 \begin{Lem}
  \label{lem:log:Sobolev}
  In $\mathbb{R}^{2}$, there exists a constant $C$ such that 
\begin{align*}
      &\|u\|_{L^\infty(B_{1/2})}^2 \leq C \| u\|_{H^1(B_{3/4})}^2  (1+\ln \frac{\|u\|_{H^2(B_{3/4})}}{\| u\|_{H^1(B_{3/4})}} ),\\
      &\|u\|_{L^\infty(B_1^c)}^2  \leq C \| u\|_{H^1(B_{5/6}^c)}^2  (1+\ln \frac{\|u\|_{H^2(B_{5/6}^c)}}{\| u\|_{H^1(B_{5/6}^c)}} ),\\
      &\|u\|_{L^{\infty}(\{\frac{1}{2}\leq |x|\leq \frac{3}{2}\} )}^2 \leq C\|u\|_{H^1(\R^2)}\|(u,\nabb u)\|_{L^2(\R^2)}\left(1+\ln\frac{\|u\|_{H^{2}(\R^2)}} {\|(u,\nabb u)\|_{L^2(\R^2)}}\right).
    \end{align*}
    Here $\nabb u=r^{-1}\Omega u=r^{-1}(x_1\pa_2-x_2\pa_1)u$. 
    \end{Lem}
 The proof is essentially based on the classical logarithmic Sobolev inequality originally introduced by Br\'{e}zis-Gallouet-Wainger. The detailed proof for the above lemma could be found for example in \cite{yang:NLW:2D}. 

 However to adapt the above Sobolev inequality to the covariant derivatives of the solution studied  in this paper, we also need a refined version of the above logarithmic Sobolev inequality.
 \begin{Lem}
  \label{lem:log:Sob:w}
  Assume that $u\in H^2(\R^2)$ such that $\|\nabla u\|_{L^2}>0$. Then it holds that 
\begin{align*}
       \|u\|_{L^\infty}^2 \leq \frac{3\sqrt{3}}{2\pi}  \| \nabla u\|_{L^2}^2  \ln \frac{\|u\|_{H^2(\mathbb{R}^2)}}{\| \nabla u\|_{L^2}}.
       \end{align*}
    \end{Lem}
\begin{proof}
For completeness we give a proof for this lemma.  Let $\hat{u}$ be the Fourier transform
\[
\hat{u}(\xi)=\frac{1}{2\pi}\int_{\R^2} u(x)e^{-ix\cdot \xi }dx.
\]
  By Fourier inverse transform, we see that
\begin{align*}
(2\pi)^2\|u\|_{L^\infty}^2\leq \|\hat{u}\|_{L^1}^2 \leq \left(\int_{\R^2} (|\xi|^4+2a |\xi|^2+b)|\hat{u}|^2 d\xi \right)\left(\int_{\R^2}(|\xi|^4+2a |\xi|^2+b)^{-1}d\xi \right).
\end{align*}
By the assumption, it is obvious that $\|u\|_{L^2}>0$. Let
\[
a=1+\frac{\|\nabla^2 u\|_{L^2}^2}{\| \nabla u\|_{L^2}^2},\quad b=\frac{\|\nabla u\|_{L^2}^2}{\|u\|_{L^2}^2}.
\]
Gagliardo-Nirenberg's inequality implies that
\[
b\leq a-1\leq \frac{a^2}{4}.
\]
In particular we can compute that
\begin{align*}
\int_{\R^2}(|\xi|^4+2a |\xi|^2+b)^{-1}d\xi=\pi \int_0^\infty \frac{1}{r^2 +2a r+b}dr=\pi \frac{1}{2\sqrt{a^2-b}}\ln \frac{a+\sqrt{a^2-b}}{a-\sqrt{a^2-b}}.
\end{align*}
Therefore we have
\begin{align*}
(2\pi)^2\|u\|_{L^\infty}^2 &\leq   \frac{\pi }{\sqrt{3} a }  \int_{\R^2} |\nabla^2 u|^2 +2a |\nabla u|^2+b |u|^2 dx \cdot \ln \frac{4a^2}{b}  \\
&\leq  6\sqrt{3}\pi  \ln \frac{\|u \|_{H^2}}{\|\nabla u\|_{L^2}} \int_{\R^2}   |\nabla u|^2  dx.
\end{align*}
The Lemma holds.
\end{proof}

To apply the above logarithmic Sobolev inequality, we need higher order weighted energy bound for the solution, which heavily relies on the potential energy decay obtained in the previous section.
\begin{Prop}
  \label{prop:phi:L2:1}
 For the case when $1<p\leq 5$, the solution $\phi$ to \eqref{eq:CSH:p:2d} verifies the following weighted energy bound 
  \begin{align*}
 (1+t)^2\|\D  \phi\|_{L^2(\mathbb{R}^2)}^2+\|(1+|t-r|)\bar{D} \phi\|_{L^2(\mathbb{R}^2)}^2+\|\phi\|_{L^2(\mathbb{R}^2)}^2&\leq C_{p} \mathcal{E}_{0, 2}
 (1+t)^{\frac{5-p}{2}}
 \end{align*}
 for some constant $C_{p}$ depending on $p$. Here recall that $\bar{D}\phi=(D_1\phi, D_2\phi)$ and  $\D=r^{-1}D_{\Omega}$.
  \end{Prop}
\begin{proof}
The above weighted energy estimate relies on the conformal energy identity \eqref{eq:Eid:conf} which also holds for the subconformal case $1<p\leq 5$. In view of the potential energy decay in Proposition \ref{prop:td:2dCSH}, we derive that 
\begin{align}
\notag
&\int_{\mathbb{R}^2}  \frac{t^2+r^2}{p+1}|\phi(t, x)|^{p+1} +\f12 \Big( |D_S\phi +\phi|^2+|tD_1\phi +x_1 D_t\phi|^2  +|tD_2\phi +x_2 D_t\phi |^2+|D_{\Omega}\phi|^2  \Big) dx \\
\notag
&= \int_{\mathbb{R}^2}  \frac{t^2+r^2}{p+1}|\phi(0, x)|^{p+1} +Q[\phi](0, x) dx +\int_0^s\int_{\mathbb{R}^2} \frac{(5-p)s}{p+1}|\phi(s, x)|^{p+1} dxds\\
\label{eq:Q:bd}
&\les \mathcal{E}_{0, 2}+\mathcal{E}_{0, 2}\int_0^t s(1+s)^{-\frac{p-1}{2}}ds\\
\notag
&\les  \mathcal{E}_{0, 2} (1+s)^{\frac{5-p}{2}}.
\end{align}
Here  we still use the notation $A\les B$ to be short for $A\leq CB$ for some constant $C$ relying only on $p$.

First note that for small $t\leq 1$, we have  
\begin{align*}
  \|\phi(t)\|_{L^2} &\leq \|\phi(0)\|_{L^2}+\int_0^1\|D_t\phi(s)\|_{L^2}ds\\
  &\les \sqrt{\mathcal{E}_{0, 2} }+ \sqrt{\mathcal{E}_{0, 0} } \les \sqrt{\mathcal{E}_{0, 2} }
\end{align*}
by using the standard energy conservation. We remark here that the initial $L^2$ bound for the solution in terms of the conformal energy is a direct consequence of a type of Hardy's inequality.

Then for large $t\geq 1$, note that 
\begin{align*}
 &\int_{\mathbb{R}^2}(D_S\phi+\phi)\cdot \phi dx=\int_{\mathbb{R}^2}(tD_t\phi+x_1 D_1\phi +x_2 D_2\phi ) \cdot\phi dx\\
 &= \frac{1}{2}\int_{\mathbb{R}^2} t\pa_t|\phi|^2+x_1 \pa_1 |\phi|^2+x_2 \pa_2 |\phi|^2+2|\phi|^2dx\\
 &=\frac{t}{2}\pa_t\int_{\mathbb{R}^2}  |\phi|^2dx +\frac{1}{2}\int_{\mathbb{R}^2}  \pa_1 (x_1|\phi|^2)+ \pa_2 (x_2|\phi|^2) dx\\
 &= \frac{t}{2}\pa_t\int_{\mathbb{R}^2}  |\phi|^2dx.
\end{align*}
Combined with the previous conformal energy bound, we in particular obtain that 
 \begin{align*}
\frac{t}{2}\frac{d}{dt}\|\phi(t)\|_{L^2}^2 &\leq \|\phi(t)\|_{L^2}\|(D_S\phi+\phi)(t)\|_{L^2}\\
&\les \|\phi(t)\|_{L^2} \sqrt{\mathcal{E}_{0, 2}} (1+t)^{\frac{5-p}{4}},
\end{align*}
from which we have 
\begin{align*}
  \|\phi(t)\|_{L^2}&\les \|\phi(1)\|_{L^2}+\int_1^t \frac{d}{dt}\|\phi(s)\|_{L^2}ds\quad   \\
   & \les \sqrt{\mathcal{E}_{0, 2} }+\sqrt{\mathcal{E}_{0, 2} } \int_1^t s^{-1}(s+1)^{\frac{5-p}{4}}ds \\
   &\les \sqrt{\mathcal{E}_{0, 2} } (t+1)^{\frac{5-p}{4}},\quad \forall t\geq 1.
\end{align*}
Therefore the $L^2$ bound for the solution   holds. 

Next we use the bound for $\|\phi\|_{L^2}$ to control the first order weighted energy, which relies on the following observation
\begin{align*}
\omega_1(tD_1\phi+x_1 D_t\phi)+\omega_2 (tD_2\phi +x_2 D_t\phi)&=t  D_r\phi +rD_t\phi,\\
\omega_2 (tD_1\phi+x_1 D_t\phi)-\omega_1 (tD_2\phi +x_2 D_t\phi)&=-\frac{t}{r} D_{\Omega}\phi,\quad \omega_i=\frac{x_i}{r},\quad i=1, 2.
\end{align*}
Note that $$D_1=\omega_1 D_r-\omega_2 r^{-1}D_{\Omega},\quad D_2=\omega_2 D_r+\omega_1 r^{-1}D_{\Omega}.$$
We therefore can bound that  
\begin{align*}
& (1+\frac{1+t^2}{r^2})|D_{\Omega}\phi|^2+(1+|t-r|^2)(|D_1\phi|^2+|D_2\phi|^2)\\
& \les (1+\frac{1+t^2}{r^2})|D_{\Omega}\phi|^2+(1+|t-r|^2)(|D_r\phi|^2+|D_{r^{-1}\Omega}\phi|^2)\\
&\les |\bar{D}\phi|^2+(1+\frac{t^2}{r^2}) |D_{\Omega}\phi|^2+|\frac{t}{t+r}(tD_r\phi +rD_t\phi)-\frac{r}{t+r}(tD_t\phi+rD_r\phi)|^2\\
&\les |\bar{D}\phi|^2+|\phi|^2+|D_S\phi+\phi|^2+|tD_1\phi+x_1 D_t\phi|^2+|tD_2\phi +x_2 D_t\phi|^2+|D_{\Omega}\phi|^2.
\end{align*}
Now by using the bound \eqref{eq:Q:bd} and the standard energy conservation, we see that 
\begin{align*}
& \int_{\mathbb{R}^2}(1+t^2+r^2)|D_{r^{-1}\Omega}\phi|^2+(1+|t-r|^2) |\bar{D}\phi|^2dx\\
&\les \int_{\mathbb{R}^2} |\bar{D}\phi|^2+|\phi|^2+|D_S\phi+\phi|^2+|tD_1\phi+x_1 D_t\phi|^2+|tD_2\phi +x_2 D_t\phi|^2+|D_{\Omega}\phi|^2 dx\\
&\les \mathcal{E}_{0, 2}  (t+1)^{\frac{5-p}{2}}.
\end{align*}
We thus finished the proof for the proposition. 
\end{proof}
To apply the logarithmic Sobolev inequality, we also need to control the growth of the second order energy. It suffices to show that the second order energy grows at most polynomially in terms of time. Unlike the flat case in \cite{yang:NLW:2D}, commuting the equation with covariant derivative may introduce new nonlinear terms. In fact we can  compute that 
\begin{align*}
\Box_A D_\mu \phi =[\Box_A, D_\mu]\phi +D_\mu \Box_A\phi =2i F_{\nu \mu}D^\nu \phi+i \pa^{\nu}F_{\nu \mu} \phi +D_\mu(|\phi|^{p-1}\phi).
\end{align*}
Since $$F_{\nu\mu}=\varepsilon_{\nu\mu\ga}J^{\ga}[\phi]=\varepsilon_{\nu\mu\ga}\Im(\phi\cdot\overline{D^\ga\phi } ),$$
we then can show that 
\begin{align*}
 \pa^{\nu}F_{\nu\mu}  &=\varepsilon_{\nu\mu\ga} \Im(D^\nu\phi \cdot \overline{D^\ga\phi})+\varepsilon_{\nu\mu\ga}  \Im(\phi \cdot \overline{D^\nu D^\ga\phi})   \\
 &=\varepsilon_{\nu\mu\ga} \Im(D^\nu\phi \cdot \overline{D^\ga\phi})+\frac{1}{2}\varepsilon_{\nu\mu\ga}  \Im(\phi \cdot \overline{D^\nu D^\ga\phi}-\phi \cdot \overline{D^\ga D^\nu\phi})  \\
 &=\varepsilon_{\nu\mu\ga} \Im(D^\nu\phi \cdot \overline{D^\ga\phi})+\frac{1}{2}\varepsilon_{\nu\mu\ga}  \Im(\phi \cdot \overline{i F^{\nu\ga}\phi} ) \\
 &=\varepsilon_{\nu\mu\ga} \Im(D^\nu\phi \cdot \overline{D^\ga\phi})-\frac{1}{2}\varepsilon_{\nu\mu\ga}   |\phi|^2  F^{\nu\ga}.
\end{align*}
In general the first term does not vanish as both  $\varepsilon_{\nu\mu\ga}$ and $\Im(D^\nu\phi \cdot \overline{D^\ga\phi})$ are skew symmetric. It seems that this new type of nonlinearity arising from commutation could not be bounded by using Gagliardo-Nirenberg inequality. To control this term, we rely on a type of  Strichartz estimate.

\bigskip 

To establish the necessary Strichartz estimate used in this paper, let's review some standard definitions in harmonic analysis. For sufficiently smooth function $f$ defined in $\R^{n}$, let  $\widehat{f}=\mathcal{F}f$   the Fourier transform of $f$. For function $P$ defined on $\R$, the differential operator $P(|\nabla |)f$ is given through the relation 
 $$P(|\nabla |)f=\mathcal{F}^{-1}(P(|\xi|)\widehat{f}(\xi)).$$
Define the homogeneous weighted Sobolev norm 
\[
\|f\|_{\dot{H}_r^{s}}=\||\nabla|^s f\|_{L^r}.
\]
The standard Sobolev space $\dot{H}^s=\dot{H}_2^s$ is the special case when $r=2$. 
We also need to use  the homogeneous Besov space $\dot{B}_{p,q}^{s} $. Recall the  Littlewood–Paley projection
 $$ P_k f=\mathcal{F}^{-1}((\varphi (\xi/2^k)-\varphi (\xi/2^{k-1}))\widehat{f}(\xi))$$
for some fixed cut-off function $\varphi$, which equals to $1$ when $|x|\leq 1$ and vanishes when $|x|>2$.  Then the Besov norm is given by 

$$ \|f\|_{\dot{B}_{r,q}^{s}}= \left(\sum\limits_{k\in\mathbb{Z}}2^{kqs}\|P_k f\|_{ L_x^r}^q \right)^{\frac{1}{q}}. 
$$
Obviously the above definition shows that 
$$  \dot{B}_{2,2}^{\rho}=\dot{H}_{2}^{\rho}=\dot{H}^{\rho},\quad \dot{H}_r^{0}=L^r  $$
since the norms are equivalent. 

We recall the following embedding 
\begin{align}
\label{eq:B1large}
& \|f\|_{\dot{H}_{r}^{\rho}}\leq C \|f\|_{\dot{B}_{r,2}^{\rho}},\quad 2\leq r<\infty,\\
\label{eq:B1small}
&  \|f\|_{\dot{B}_{r,2}^{\rho}}  \leq C\|f\|_{\dot{H}_{r}^{\rho}},\quad  1< r\leq2,\\
\label{eq:B2}
&\|f\|_{\dot{B}_{r_1,s}^{\rho_1}}\leq C\|f\|_{\dot{B}_{r_2,s}^{\rho_2}},\quad 1\leq r_2\leq r_1\leq\infty,\ \frac{n}{r_2}-\frac{n}{r_1}=\rho_2-\rho_1,
\end{align}
which have been shown for example in \cite{velo95:Strichartz:NW}. In particular for the case when $1<r<2$ and $n=2$, we have the following trace inequality
\begin{align}
\label{eq:B3}
&\|f\|_{\dot{H}^{1-2/r}}\leq C\|f\|_{\dot{B}_{2,2}^{1-2/r}}\leq C\|f\|_{\dot{B}_{r,2}^{0}}\leq C\|f\|_{L^r}.
\end{align}
Moreover by using H\"older inequality, we also have  the following type of  Gagliardo-Nirenberg interpolation
\begin{align}
\label{eq:B4}
&\|f\|_{\dot{B}_{r,s}^{\rho}}\leq \|f\|_{\dot{B}_{r_1,s}^{\rho_1}}^{\theta}\|f\|_{\dot{B}_{r_2,s}^{\rho_2}}^{1-\theta},\quad  \frac{1}{r}=\frac{\theta}{r_1}+\frac{1-\theta}{r_2},\ \rho=\theta\rho_1+(1-\theta)\rho_2,\ \theta\in(0,1).
\end{align}
Now let's recall the classical Strichartz estimate in space dimension $2$ for the half wave equation. 
\begin{Lem}[Strichartz Estimate]
  \label{lem:Strichartz2}

  Assume that $2\leq q,\rho\leq \infty$ such that 
$ \frac{2}{q}+\frac{1}{\rho}\leq  \frac{1}{2}$. Let $b=1-\frac{2}{\rho} -\frac{1}{q}$ and $s\in\R$.  Then for function $f$ defined on $\R^2$ it holds that 
\begin{align*}
      &\|\exp(\pm it|\nabla |)f\|_{L_t^{q}\dot{B}_{\rho,2}^{-s}}\leq C\|f\|_{\dot{H}^{b-s}}.
    \end{align*}
\end{Lem}
See proof for example in  \cite{velo95:Strichartz:NW} or \cite{Tao98:endStri}. Based on the classical Strichartz estimate, we first establish a type of Strichartz estimate  for linear  half wave equation in $\mathbb{R}^{1+2}$.

\begin{Lem}
  \label{lem:Strichartz1}
  Assume that $\frac{8}{5}<r<2$ and $\frac{2}{r}-\frac{1}{4}<s<1$. Let $u(t, x)$ be solution to the linear half wave equation
\[
(\partial_t\pm i|\nabla|)u=f
\]
   in $\mathbb{R}^{1+2}$.
Then there exists a constant $C$, independent of $f$ and $T$, such that 
\begin{align*}
      &\|u\|_{L^{4}([T,T+1];\dot{B}_{\infty,2}^{-s})}\leq C(\|u\|_{L^{\infty}([T,T+1];L^2)}+\|f\|_{L^{1}([T,T+1];L^r)}).
    \end{align*}
    Here the $L^r$ norm and $\dot{B}_{\infty,2}^{-s}$ is taken with respect to the variable $x\in \R^2$. 
\end{Lem}
\begin{proof}
By Duhamel's formula we have 
\begin{align*}
u(T+t)=\exp(\mp it|\nabla |)u(T)+\int_0^t\exp(\mp i(t-s)|\nabla |)f(T+s)ds=:u_1(T+t)+u_2(T+t).
\end{align*}
Apply Lemma \ref{lem:Strichartz2} with 
\[
\rho=\infty,\quad b=s,\quad  q=\frac{1}{1-s}.
\]
The assumption on $r$ and $s$ shows that $q>4$. Then by using 
 H\"older inequality, we deduce that  
\begin{align*}
      &\|u_1\|_{L^{4}([T,T+1];\dot{B}_{\infty,2}^{-s})}\leq\|u_1\|_{L_t^{q}\dot{B}_{\rho,2}^{-s}}=\|\exp(it|\nabla |)u(T)\|_{L_t^{q}\dot{B}_{\rho,2}^{-s}}\leq C\|u(T)\|_{L^2}.
\end{align*}
To bound $u_2$, we still rely on the  Lemma \ref{lem:Strichartz2}  with 
\[
\rho=\infty,\quad b=s+1-\frac{2}{r},\quad \frac{1}{q}=\frac{2}{r}-s. 
\]
Since $r<2$ and $s<1$, we have $0<\frac{2}{r}-s<\frac{1}{4}$, which indicates that $q>4$. 
Hence by using H\"older inequality and the embedding \eqref{eq:B3}, we can estimate that
\begin{align*}
      \|\exp(\mp i(t-s)|\nabla|)f(T+s)\|_{L^{4}([s,1];\dot{B}_{\infty,2}^{-s})} & \leq\|\exp(\mp i(t-s)|\nabla |)f(T+s)\|_{L_t^{q}\dot{B}_{\rho,2}^{-s}}\\ 
      &\leq  C\|f(T+s)\|_{\dot{H}^{b-s}}=C\|f(T+s)\|_{\dot{H}^{1-2/r}} \\ 
      &\leq C\|f(T+s)\|_{L^{r}},\quad 0<s<1.
    \end{align*} 
    Then   Minkowski inequality leads to 
    \begin{align*}
         \|u\|_{L^{4}([T,T+1];\dot{B}_{\infty,2}^{-s})} & \leq\|u_1\|_{L^{4}([T,T+1];\dot{B}_{\infty,2}^{-s})}+\|u_2\|_{L^{4}([T,T+1];\dot{B}_{\infty,2}^{-s})}\\
      & \leq C\|u(T)\|_{L^2}+ C\|f\|_{L^{1}([T,T+1];L^r)} \\ 
      &\leq C(\|u\|_{L^{\infty}([T,T+1];L^2)}+\|f\|_{L^{1}([T,T+1];L^r)}).
    \end{align*}
    This completes the proof for the Lemma.
\end{proof}

Next we adapt the above Strichartz estimate for the half wave equation to the linear transport equation for the scalar field under Coulomb gauge. 
\begin{Lem}
  \label{lem:Strichartz}
  Let $u=(u_0, u_1, u_2)$ be defined on $\R^{1+2}$. Denote 
  \[
f_0=\partial_tu_0-\partial_1u_1-\partial_2u_2,\quad f_{\mu \nu }=\partial_{\mu} u_{\nu}-\partial_{\nu} u_{\mu},\quad \mu, \nu=0, 1, 2.
  \]
 Assume that $\frac{8}{5}<r<2$ and $\frac{2}{r}-\frac{1}{4}<s<1$. Then it holds that 
\begin{align*}
      &\|u\|_{L^{4}([T,T+1];\dot{B}_{\infty,2}^{-s})}\leq C(\|u\|_{L^{\infty}([T,T+1];L^2)}+\|(f_0,f_{01},f_{02},f_{12})\|_{L^{1}([T,T+1];L^r)}).
    \end{align*}
    for some constant $C$ which is independent of $T$ and $u$. 
\end{Lem}
 \begin{proof}
 To reduce this lemma to the previous lemma, we rely on the Riesz transform
  $$R_j f=\mathcal{F}^{-1}(-i\xi_j |\xi|^{-1}\widehat{f}(\xi) ), j=1,2.$$
  In particular by definition we have 
  $$|\nabla |R_j f=-\partial_j f, \quad R_1^2f+R_2^2f=-f.$$
  Moreover since the operators $R_1,R_2,|\nabla |,\partial_1,\partial_2$ are commutative, we deduce that 
  $$ R_1\partial_1f+R_2\partial_2f=|\nabla|f .$$
  Regarding the Riesz transform $R_j$, it  is bounded on $L^r$ for $1<r<\infty$ (see for example chapter one in \cite{stein:harmonic}). As the Besov space $\dot{B}_{p,q}^{s} $ is homogeneous, the Riesz transform is also bounded on it for $1\leq p\leq\infty$ (it suffices to consider the case when $\hat{f}$ is supported on $|\xi|\in [\frac{1}{2}, 2]$). 
 
Now let 
 $$h_1=R_1u_1+R_2u_2, \quad h_2=-R_1u_2+R_2u_1.$$
Then we get 
 \begin{align*} 
 u_1=-R_1h_1-R_2h_2, &\quad  u_2=-R_2h_1+R_1h_2, \\ 
  |\nabla|h_1=-\partial_1u_1-\partial_2u_2,&\quad 
|\nabla |h_2=\partial_1u_2-\partial_2u_1=f_{12}.
\end{align*}
In particular, we derive that 
$$ \partial_tu_0+|\nabla |h_1=f_0.$$
On the other hand, we compute that 
\begin{align*}
\partial_t h_1=R_1\partial_t u_1+R_2 \partial_t u_2= R_1(\partial_1u_0+f_{01})+R_2(\partial_2u_0+f_{02})=|\nabla |u_0+f_3,
\end{align*}
in which we denote $f_3=R_1 f_{01}+R_2 f_{02}.$ Therefore we conclude that 
$$(\partial_t \mp i|\nabla|)(u_0\pm i h_1)=f_0\pm if_3.  $$
Lemma \ref{lem:Strichartz1} then implies that 
 \begin{align*}
      &\|u_0\pm ih_1\|_{L^{4}([T,T+1];\dot{B}_{\infty,2}^{-s})}\leq C(\|u_0\pm ih_1\|_{L^{\infty}([T,T+1];L^2)}+\|f_0\pm if_3\|_{L^{1}([T,T+1];L^r)}).
    \end{align*}
Since the Riesz transforms  $R_1,R_2$ are bounded on $L^2$ and $L^r$, we therefore conclude that 
     \begin{align*}
      &\|(u_0,h_1)\|_{L^{4}([T,T+1];\dot{B}_{\infty,2}^{-s})} \\ 
      &\leq C(\|(u_0,h_1)\|_{L^{\infty}([T,T+1];L^2)}+\|(f_0,f_3)\|_{L^{1}([T,T+1];L^r)}) \\ 
      & \leq C(\|(u_0,R_1 u_1+R_2 u_2)\|_{L^{\infty}([T,T+1];L^2)}+\|(f_0, R_1 f_{01}+R_2 f_{02})\|_{L^{1}([T,T+1];L^r)}) \\ 
      &\leq C(\|(u_0,u_1,u_2)\|_{L^{\infty}([T,T+1];L^2)}+\|(f_0,f_{01},f_{02})\|_{L^{1}([T,T+1];L^r)}).
    \end{align*}
Now it  remains to estimate $h_2$. Recall that $ |\nabla |h_2=f_{12}$. By using the embedding   \eqref{eq:B3}, we obtain that 
\begin{align*}
      &\|h_2\|_{\dot{H}^{2-2/r}}= \||\nabla |h_2\|_{\dot{H}^{1-2/r}}= \|f_{12}\|_{\dot{H}^{1-2/r}}\leq C\|f_{12}\|_{L^{r}}.
    \end{align*}
    Since   the Riesz transforms  $R_1,R_2$ are bounded on $L^2$, we then conclude that 
    \begin{align*}
      &\|h_2\|_{L^{\infty}([T,T+1];L^2)}= \|-R_1u_2+R_2u_1 \|_{L^{\infty}([T,T+1];L^2)}\leq C\|(u_1,u_2)\|_{L^{\infty}([T,T+1];L^2)}.
    \end{align*}
The assumptions on $r$ and  $s$ show that 
$$0<1-s<1-(\frac{2}{r}-\frac{1}{4})=\frac{5}{4}-\frac{2}{r}<2-\frac{2}{r}.$$
Hence by using the embedding  \eqref{eq:B2}, \eqref{eq:B4} and \eqref{eq:B1small}, we obtain that 
\begin{align*}
      &\|h_2\|_{\dot{B}_{\infty,2}^{-s}}\leq C\|h_2\|_{\dot{B}_{2,2}^{1-s}}\leq C\|h_2\|_{\dot{B}_{2,2}^{0}}^{1-\theta}\|h_2\|_{\dot{B}_{2,2}^{2-2/r}}^{\theta}\leq C\|h_2\|_{L^2}^{1-\theta}\|h_2\|_{\dot{H}^{2-2/r}}^{\theta},
    \end{align*}
    in which 
    $$ 0<\theta=\frac{1-s}{2-\frac{2}{r}}<\frac{\frac{5}{4}-\frac{2}{r} }{2-\frac{2}{r}}\leq \frac{1}{4}. $$  
   Then by using H\"older inequality we can show that 
   \begin{align*}
      \|h_2\|_{L^{4}([T,T+1];\dot{B}_{\infty,2}^{-s})}&\leq\|h_2\|_{L^{1/\theta}([T,T+1];\dot{B}_{\infty,2}^{-s})} \\ 
      &\leq C\|h_2\|_{L^{\infty}([T,T+1];L^2)}^{1-\theta}\|h_2\|_{L^{1}([T,T+1];\dot{H}^{2-2/r})}^{\theta}\\
      &\leq  C\|(u_1,u_2)\|_{L^{\infty}([T,T+1];L^2)}^{1-\theta}\|f_{12}\|_{L^{1}([T,T+1];L^{r})}^{\theta} \\ 
      & \leq C(\|(u_1,u_2)\|_{L^{\infty}([T,T+1];L^2)}+\|f_{12}\|_{L^{1}([T,T+1];L^{r})}).
    \end{align*}
    Finally as $u_1=-R_1h_1-R_2h_2,$ $u_2=-R_2h_1+R_1h_2$ and $R_1,R_2$ are bounded on $\dot{B}_{\infty,2}^{-s}$, we conclude that 
    \begin{align*}
      \|(u_0,u_1,u_2)\|_{L^{4}([T,T+1];\dot{B}_{\infty,2}^{-s})} & \leq C\|(u_0,h_1,h_2)\|_{L^{4}([T,T+1];\dot{B}_{\infty,2}^{-s})}\\ 
      &\leq C(\|(u_0,u_1,u_2)\|_{L^{\infty}([T,T+1];L^2)}+\|(f_0,f_{01},f_{02},f_{12})\|_{L^{1}([T,T+1];L^r)}).
    \end{align*}
    This completes the proof.
     \end{proof}

With the above preparations, we are ready to show that the second order energy grows at most polynomially. For simplicity in the following, we allow the implicit constant in $B\les C$ also rely on the initial conformal energy $ \mathcal{E}_{0, 2}$ and the first order initial energy $\mathcal{E}_{1, 0}$.

  \begin{Prop}
  \label{prop:PD:H2}
The second order energy of the solution $\phi$ grows at most polynomially
  \begin{align*}
   \|D_j D_k\phi\|_{L^2(\mathbb{R}^2)}&\les  (1+t)^{5},\quad j, k=1, 2.
 \end{align*}
  \end{Prop}
\begin{remark}
 We remark here that Huh in \cite{Huh07:CSH} obtained  control of $ H^2$ norm of the solution but the norm may grow exponentially. 
\end{remark}

\begin{proof}
To make use of the Strichartz estimate established above, we work under a particular gauge. Choose the Coulomb gauge condition 
 $$ \partial_1A_1+\partial_2A_2=0,$$
under which the connection field $A$ verifies the equation
\begin{align*}
&\Delta A_0=-\partial_1F_{01}-\partial_2F_{02},\quad \Delta A_1=-\partial_2F_{12},\quad \Delta A_2=\partial_1F_{12}.
\end{align*}
By the Gagliardo-Nirenberg inequality and in view of the energy conservation \eqref{eq:E3}, we first have the bound  
 \begin{align*}
&\|\phi\|_{L^{q_1}}^{q_1} \les \|\phi\|_{L^{p+1}}^{p+1}\|D\phi\|_{L^2}^{q_1-p-1}\les 1,\quad \forall p+1\leq q_1<+\infty.
\end{align*}
Then by using  the standard elliptic $L^p$ estimate (see for example \cite{elliptic}),
 we derive that 
  \begin{align*}
&\|A\|_{L^q}\les \|\nabla A\|_{L^r}\les \|F\|_{L^r}\les \|\phi\|_{L^q} \|D\phi\|_{L^2}\les 1,\quad \forall\ 1<r<2,\quad \forall p+1\leq q <\infty,\quad  \frac{1}{q}=\frac{1}{r}-\frac{1}{2}.
\end{align*}
This shows that  
\begin{align*}
 &\||A||D\phi|\|_{L^r}\leq \|A\|_{L^q}\|D\phi\|_{L^2}\les 1,\quad \forall p+1\leq q<+\infty,\quad  \frac{1}{q}=\frac{1}{r}-\frac{1}{2}, \\ 
 &\|\phi F\|_{L^r}\leq \| |D\phi| |\phi|^2\|_{L^r}\leq \|D\phi\|_{L^2}\|\phi\|_{L^{2q}}^2 \les 1,\\ 
 & \| |\phi|^{p-1}\phi\|_{L^r}=\|\phi\|_{L^{pr}}^p\les 1, \quad \frac{p+1}{p}<r<2. 
 \end{align*}
We now write the system \eqref{eq:CSH:p:2d} as transport equations for $D\phi $ 
\begin{align*}
&D_t^2\phi=D_1^2\phi+D_2^2\phi-|\phi|^{p-1}\phi,\ D_\mu D_\nu\phi=D_\nu D_\mu \phi+iF_{\mu \nu }\phi,\Leftrightarrow\\
&\partial_t D_t\phi-\partial_1D_1\phi-\partial_2D_2\phi=-iA_0D_t\phi+iA_1D_1\phi+iA_2D_2\phi-|\phi|^{p-1}\phi:=f_0,\\
&\partial_\mu D_\nu \phi-\partial_\nu D_\mu \phi=-iA_\mu D_\nu\phi+iA_\nu D_\mu \phi+iF_{\mu \nu }\phi:=f_{\mu \nu }.
\end{align*}
The above argument shows that 
\begin{align*}
\|(f_0, f_{01}, f_{02}, f_{12})\|_{L^r}\les 1, \quad \max\{\frac{p+1}{p}, \frac{2(p+1)}{p+3}\}<r<2. 
\end{align*}
Now let  $r$ be sufficiently close to $2$ and $s$  slightly larger than $\frac{3}{4}$. The Strichartz estimate of Lemma \ref{lem:Strichartz} then shows that 
 \begin{align*}
\|D\phi\|_{L^{4}([T,T+1];\dot{B}_{\infty,2}^{-s})}\les 1,\quad \forall T\geq 0. 
\end{align*}
In particular we conclude that 
\begin{align}
\label{eq:dphiint}
  \int_0^{T}\|D\phi\|_{\dot{B}_{\infty,2}^{-s}}^{q_2} dt\les T+1 ,\quad \forall  0<q_2<4,\quad  T\geq 0.
\end{align}
Let $\epsilon>0$ be sufficiently small such that $s(1+\epsilon)<1$. By using the embedding  \eqref{eq:B1small} and \eqref{eq:B4}, we can show that 
\begin{align*}
 \||D\phi|^2 |\phi |\|_{L^2} &\leq \|D\phi\|_{L^{ 4+2\epsilon}}^2\|\phi\|_{L^{\frac{2(2+\epsilon)}{\epsilon}}} \les   \|D\phi\|_{\dot{B}_{4+2\epsilon,2}^{0}}^2\\
 &\les  \|D\phi \|_{\dot{B}_{\infty,2}^{-s}}^{\frac{2+2\epsilon}{2+\epsilon}}\|D\phi\|_{\dot{B}_{2,2}^{s(1+\epsilon)}}^{\frac{2}{2+\epsilon} } \\ 
 & \les  \|D\phi \|_{\dot{B}_{\infty,2}^{-s}}^{\frac{2+2\epsilon}{2+\epsilon}} \|D\phi\|_{\dot{H}^{s( 1+\epsilon)}}^{ \frac{2}{2+\epsilon} }\\ 
 & \les \|D\phi\|_{\dot{B}_{\infty,2}^{-s}}^{ \frac{2+2\epsilon}{2+\epsilon} }\|D\phi\|_{L^2}^{\frac{2(1-s(1+\epsilon ))}{2+\epsilon}   }\|D\phi\|_{\dot{H}^{1}}^{ \frac{2 s(1+\epsilon )}{2+\epsilon}  } \\ 
 & \les \|D\phi\|_{\dot{B}_{\infty,2}^{-s}}^{ \frac{2+2\epsilon}{2+\epsilon} } \|D\phi\|_{\dot{H}^{1}}^{ \frac{2 s(1+\epsilon )}{2+\epsilon}  }.
\end{align*}
 We also have\begin{align*}
  &\|D\phi\|_{\dot{H}^{1}}=\|\nabla  D\phi\|_{L^2}\leq \|D^2\phi\|_{L^2}+\||A||D\phi|\|_{L^2}\leq \|D^2\phi\|_{L^2}+\|A\|_{L^{p+1}}\|D\phi\|_{L^{2(p+1)/({p-1})}}\\
  &\|D\phi\|_{L^{2(p+1)/({p-1})}}\lesssim\|\nabla  |D\phi|\|_{L^2}+\|D\phi\|_{L^2}\leq\|D^2\phi\|_{L^2}+\|D\phi\|_{L^2},\quad \|A\|_{L^{p+1}}\lesssim1.
\end{align*}Thus $\|D\phi\|_{\dot{H}^{1}}\lesssim\|D^2\phi\|_{L^2}+\|D\phi\|_{L^2} $ and\begin{align}\label{D1}
 \||D\phi|^2 |\phi |\|_{L^2}
 & \les \|D\phi\|_{\dot{B}_{\infty,2}^{-s}}^{ \frac{2+2\epsilon}{2+\epsilon} } \|D\phi\|_{\dot{H}^{1}}^{ \frac{2 s(1+\epsilon )}{2+\epsilon}  }\les \|D\phi\|_{\dot{B}_{\infty,2}^{-s}}^{ \frac{2+2\epsilon}{2+\epsilon} } (\|D^2\phi\|_{L^2}+\|D\phi\|_{L^2})^{ \frac{2 s(1+\epsilon )}{2+\epsilon}  }.
\end{align}
Note that 
\begin{align*}
|\Box_A D\phi | & \leq 2|F||D\phi|+(|D\phi|^2+|F||\phi|^2)|\phi|+p |D\phi||\phi |^{p-1} \\ 
&\leq 3 |D\phi|^2|\phi|+|D\phi|(|\phi|^4+p|\phi|^{p-1}).
\end{align*}
The energy identity \eqref{eq:energy:id} (without the nonlinearity $|\phi|^{p-1}\phi$) applied to $D\phi$ for the vector field $X=\pa_t$ and the function $\chi=0$  then gives
\begin{align*}
\int_{\mathbb{R}^2} |D D_\mu\phi|^2 dx = \int_{\mathbb{R}^2} |D D_\mu\phi(0, x)|^2 dx-2\int_0^t \int_{\mathbb{R}^2}  \l \Box_A D_\mu\phi,  D_t D_\mu \phi \r - F_{\ga 0} J^\ga[D_\mu\phi]  dxdt.
\end{align*}
By definition we can bound 
\begin{align*}
|F_{\ga 0} J^\ga[D_\mu\phi]|\leq |D\phi||\phi| |D\phi||DD_\mu \phi |.
\end{align*}
We then derive the energy estimate for $D\phi $ 
\begin{align*}
\frac{d}{dt}\|D D_\mu\phi\|_{L^2}^2&\leq 2\int_{\mathbb{R}^2} |DD_\mu\phi|(4 |D\phi|^2|\phi|+|D\phi|(|\phi |^4+p|\phi |^{p-1})) dx\\
&\les \| DD_\mu\phi \|_{L^2} \| |D\phi|^2 |\phi |+ |D\phi| (|\phi|^4+|\phi |^{p-1}) \|_{L^2},
\end{align*}
which in particular implies that 
\begin{align*}
 \|D D_\mu\phi\|_{L^2} \les 1+\int_0^t   \| |D\phi|^2 |\phi |+ |D\phi| (|\phi|^4+|\phi |^{p-1}) \|_{L^2} dt.
\end{align*}
Now by using 
 Gagliardo-Nirenberg inequality and for sufficiently small $\epsilon$ depending only on $p$, we can bound that (for $q=4$ or $p-1$)
\begin{align*}
\||D\phi|\cdot |\phi|^{q}\|_{L^2} &\leq \|D\phi\|_{L^{2+\ep}}\|\phi\|_{L^{\frac{2(2+\ep)q}{\ep}}}^{q}\\
&\les \|D\phi\|_{L^2}^{\frac{2}{2+\ep}}\|DD\phi\|_{L^2}^{\frac{\ep}{2+\ep}} \|\phi\|_{L^{p+1}}^{\frac{(p+1)\ep}{2(2+\ep)}}\|D\phi\|_{L^2}^{q-\frac{(p+1)\ep}{2(2+\ep)}}\\
&\les \|DD\phi\|_{L^2}^{\frac{\ep}{2+\ep}}.
\end{align*}
 Combining the previous bound  \eqref{D1}  for $|D\phi|^2 |\phi|$, we then obtain 
\begin{align*}
 \|DD\phi\|_{L^2}&\les 1+\int_0^t   (\| |D\phi|^2 |\phi |\|_{L^2}+ \||D\phi|\cdot |\phi|^{p-1}\|_{L^2}+\||D\phi|\cdot |\phi|^{4}\|_{L^2})dt\\
 &\les 1+\int_0^t   \|D\phi\|_{\dot{B}_{\infty,2}^{-s}}^{\frac{2+2\epsilon }{2+\epsilon }}(\|D^2\phi\|_{L^2}+\|D\phi\|_{L^{2}})^{\frac{2}{2+\epsilon}s(1+\epsilon)}+\|DD\phi\|_{L^2}^{\frac{\ep}{2+\ep}}dt.
\end{align*}
In view of the bound \eqref{eq:dphiint} with $q_2=\frac{2+2\epsilon}{2+\epsilon }$, we then conclude that 
\begin{align*}
\|DD\phi\|_{L^2}\les 1+ (1+t)^{\frac{2+\epsilon}{2+\epsilon-2s(1+\epsilon) }}.
\end{align*}
The proposition then follows by taking $s=\frac{3}{4}+\epsilon $  with  $\epsilon$  sufficiently small.

 \end{proof}

Consequently, in view of the logarithmic Sobolev inequality, we now show that the solution decays inverse polynomially in time for all $p>1$.
\begin{Prop}
\label{prop:ptdecay:allp}
For $1<p\leq 5$, the solution $\phi$ to \eqref{eq:CSH:p:2d} verifies the following decay estimate
\begin{align*}
|\phi|\leq C \sqrt{\ln(2+t)}
(1+t)^{-\frac{p-1}{8} }
\end{align*}
for some constant $C$ depending on $\mathcal{E}_{0, 2}$ and $\mathcal{E}_{1, 0}$. 
\end{Prop}
\begin{proof}
First of all, by using Lemma \ref{lem:log:Sob:w} we derive a rough bound for $|\phi|$. 
By considering the gauge invariant norm $|\phi|^2$ and the energy conservation, we first show that 
\begin{align*}
\int_{\R^2} |\nabla |\phi|^2|^2 dx=4\int_{\R^2} |\l \phi, \bar{D}\phi \r |^2 dx\leq 4 \|\phi\|_{L^\infty}^2 \int_{\R^2}|\bar{D}\phi |^2 dx\les  \|\phi\|_{L^\infty}^2.
\end{align*}
For the second order norm, in view of Proposition \ref{prop:PD:H2}, Proposition \ref{prop:phi:L2:1} and the Gagliardo-Nirenberg inequality, we can bound that 
\begin{align*}
\int_{\R^2}|\nabla^2 |\phi|^2|^2+|\phi|^4 dx &\les \int_{\R^2} |\bar{D}\phi|^4+|\phi|^2 (|\bar{D}\bar{D}\phi|^2+|\bar{D}\phi |^2)dx\\
&\les  \int_{\R^2} |\bar{D}\bar{D}\phi|^2 dx \int_{\R^2} |\bar{D}\phi|^2 dx +\|\phi\|_{L^\infty}^2 \int_{\R^2} |\bar{D}\bar{D}\phi|^2+|\phi|^2dx \\
&\les (1+t)^{10} (1+\|\phi\|_{L^\infty}^2). 
\end{align*}
Lemma \ref{lem:log:Sob:w} then implies that 
\begin{align*}
\||\phi |^2\|_{L^\infty}^2=\|\phi  \|_{L^\infty}^4 & \les  \| \nabla |\phi|^2\|_{L^2}^2 \max(10\ln (1+t)+\ln(1+\|\phi\|_{L^\infty}^2)-\ln \| \nabla |\phi |^2\|_{L^2}^2,1 )\\
&\les \| \phi \|_{L^\infty }^2 (\ln (2+t)+\ln(1+1/\|\phi\|_{L^\infty}) ), 
\end{align*}
which leads to the rough bound 
\[
\|\phi\|_{L^\infty}\les \sqrt{\ln(2+t)}.
\]
Next we improve this bound to time decay. Since  the estimate of the Proposition holds trivially when $t \leq 8$, for fixed $t >8$ and for the case when $|x|\leq \frac{t}{2}$, let
\begin{align*} 
\tilde{\phi}(x)=|\phi(t, tx)|^2.
\end{align*}
In view of the weighted energy estimate of Proposition \ref{prop:phi:L2:1}, we conclude that
\begin{align*}
\int_{|x|\leq \frac{3t}{4}}|\bar{D}\phi|^2dx\leq 4^2t^{-2}\int_{|x|\leq t} (t-r)^2 |\bar{D}\phi|^2 dx \les 
(1+t)^{\frac{1-p}{2}}.
\end{align*}
Note that
\begin{align*}
\pa_{j}\tilde{\phi}(x)&=2t\l \phi(t, tx), 
 D_j\phi (t, tx) \r, \\ 
   \pa_j\pa_k \tilde{\phi}(x)&=2t^2 \l D_k\phi(t, tx), D_j\phi (t, tx) \r+2t^2\l \phi(t, tx), D_k D_j\phi (t, tx) \r
\end{align*}
We therefore conclude that 
\begin{align*}
\int_{|x|\leq \frac{3}{4}}|\tilde{\phi}|^2+|\nabla\tilde{\phi}|^2dx & \les \int_{|x|\leq  \frac{3t}{4} }( t^{-2}|\phi|^2 +|\bar{D}\phi|^2) |\phi|^2 dx \\
&\les
\|\phi \|_{L^\infty}^2 (1+t)^{\frac{1-p}{2}} \\
&\les (1+t)^{\frac{1-p}{2}}\ln(1+t).
\end{align*}
By using Gagliardo-Nirenberg inequality, for the second order derivative, in view of the first order energy bound of Proposition \ref{prop:PD:H2} we estimate that 
\begin{align*}
\int_{|x|\leq \frac{3}{4} }|\nabla^2 \tilde{\phi}|^2 dx & \les t^2\int_{|x|\leq \frac{3t}{4} } |\bar{D} \phi|^4+ |\phi|^2 |\bar{D}^2\phi |^2 dx \\ 
& \les   t^2 \int_{|x|\leq \frac{3t}{4} } |\bar{D}\bar{D}\phi|^2dx \int_{|x|\leq \frac{3t}{4} } |\bar{D}\phi |^2dx + (1+t)^{12} \|\phi\|_{L^\infty}^2\\ 
&\les  (1+t)^{12} \ln(2+t).
\end{align*}
Then the logarithmic Sobolev embedding of Lemma \ref{lem:log:Sobolev} implies that 
\begin{align*}
\|\tilde{\phi}\|_{L^\infty(\{|x|\leq \frac{1}{2}\})}^2 =\|\phi\|^{4}_{L^\infty(\{|x|\leq \frac{t}{2}\})} & \les  \|\tilde\phi\|_{H^1(\{|x|\leq \frac{3}{4}\})}^2 \max (13\ln(1+t)-\ln \|\tilde\phi\|_{H^1(\{|x|\leq \frac{3}{4}\})}^2,1 ) \\
&\les (1+t)^{\frac{1-p}{2}} \ln(2+t)(1+\ln(1+t)), 
\end{align*}
from which we can conclude that 
\[
\|\phi\|_{L^\infty(\{|x|\leq \frac{t}{2}\}) } \les (1+t)^{\frac{1-p}{8}} \sqrt{\ln (1+t)} . 
\]
Here by our convention the implicit constant   relies on  the initial energy $\mathcal{E}_{0, 2}$ and $\mathcal{E}_{1, 0}$.

\bigskip 

Decay estimate for the case when  $|x|\geq \frac{3}{2}t$ is similar. Let $r_0=\frac{3t}{2}$ and
\begin{align*}
\psi(x)=|\phi(t, r_0x)|^2,\quad |x|\geq 1.
\end{align*}
For the $H^1$ norm we can estimate that 
\begin{align*}
\int_{|x|\geq \frac{5}{6} } |\psi|^2+|\nabla\psi|^2 &\les   r_0^{-2} \int_{\mathbb{R}^2} (|\phi|^2 +|(r-t)\bar{D}\phi|^2 )|\phi|^2 dx\\
&\les  r_0^{-2} (1+t)^{\frac{5-p}{2}} \ln (1+t) \\ 
& \les    (1+r_0)^{\frac{1-p}{2}} \ln (1+r_0).
\end{align*}
Similarly for the second order derivative of $\psi$, we estimate that 
\begin{align*}
&\int_{|x|\geq \frac{5}{6} }|\nabla^2 \psi|^2 dx  \leq r_0^2 \int_{|x|\geq \frac{5r_0}{6}} |\bar{D} \phi|^4+|\phi|^2|\bar{D}\bar{D}\phi |^2 dx  \\
 & \les r_0^2 \int_{|x|\geq \frac{5r_0}{6}} |\bar{D} \phi|^2  dx \int_{|x|\geq \frac{5r_0}{6}} (|\bar{D}\bar{D} \phi|^2+ |\bar{D} \phi|^2) dx+r_0^2    \ln (1+t) \int_{|x|\geq \frac{5r_0}{6}} |\bar{D}\bar{D} \phi|^2  dx\\
 &\les (1+r_0)^{12} \ln (e+r_0).  
\end{align*}
Therefore in view of  the logarithmic Sobolev embedding of Lemma \ref{lem:log:Sobolev}, we conclude  that
\begin{align*}
 \|\psi\|_{L^\infty(\{|x|\geq 1 \})}^2=\|\phi\|_{L^\infty(\{|x|\geq  r_0 \})}^4 & \leq \|\psi\|_{H^1(\{|x|\geq \frac{5 }{6}\})}^2\max(13\ln(1+r_0)-\ln \|\psi\|_{H^1(\{|x|\geq \frac{5 }{6}\})}^2,1 )\\
 &\les  \ln(1+r_0) (1+r_0)^{\frac{1-p}{2}} \ln(e+r_0). 
\end{align*}
This implies that 
\begin{align*}
\|\phi\|_{L^\infty(\{|x|\geq \frac{3}{2}t \})}=\|\phi\|_{L^\infty(\{|x|\geq  r_0 \})} \les  (1+r_0)^{\frac{1-p}{8}} \sqrt{\ln(1+r_0)}\les (1+t)^{\frac{1-p}{8}} \sqrt{\ln(1+t)}.
\end{align*}
Finally on the region near the light cone $\{\frac{1}{2}t\leq |x| \leq \frac{3}{2}t\}$, we still consider  the function $\widetilde{\phi}=|\phi(t, tx)|^2 $ with fixed $t>8$. 
Let 
\[
\kappa =\|\phi\|_{L^\infty(\{\frac{1}{2}t\leq |x|\leq \frac{3}{2}t\})}.
\]
By using the weighted energy estimate of  Proposition \ref{prop:phi:L2:1} and the first order energy bound of Proposition \ref{prop:PD:H2} together with the above pointwise decay estimate for $\phi$ when $|r-t|\geq \frac{t}{2}$, we can show that
\begin{align*}
\|\widetilde{\phi}\|_{L^2(\R^2)}&  \les t^{-1}\|\phi \|_{L^2(\R^2)}  \|\phi \|_{L^\infty(\R^2)}\les  
(1+t)^{\frac{1-p}{4}}(\kappa +\sqrt{\ln(1+t)} (1+t)^{\frac{1-p}{8}}),\\
 \|\nabla\widetilde{\phi}\|_{L^2(\R^2)} &\leq 2\| |\bar{D}\phi ||\phi |\|_{L^2(\R^2)} \les \kappa +\sqrt{\ln(1+t)} (1+t)^{\frac{1-p}{8}}, \\ 
\|\nabb\widetilde{\phi}\|_{L^2(\R^2)} & \leq 2 \| |\phi| |D_{r^{-1}\Omega}\phi |\|_{L^2(\R^2)} \les
(1+t)^{\frac{1-p}{4}} (\kappa +\sqrt{\ln(1+t)} (1+t)^{\frac{1-p}{8}}),\\
 \|\pa_j\pa_k \tilde{\phi}\|_{L^2(\R^2)} & \les  t(\| \bar{D} \phi  \|_{L^4(\R^2)}^2+\||\phi| |\bar{D}\bar{D}\phi|\|_{L^2(\R^2)})\\
 &\les t(\| \bar{D} \phi  \|_{L^2(\R^2)}  \|\bar{D} \bar{D} \phi  \|_{L^2(\R^2)}  +\||\phi| |\bar{D}\bar{D}\phi|\|_{L^2(\R^2)})\\
 &\les   (1+t)^6(1+\kappa +\sqrt{\ln(1+t)} (1+t)^{\frac{1-p}{8}} )  
\end{align*}
Then  the logarithmic Sobolev embedding of Lemma \ref{lem:log:Sobolev} implies that 
\begin{align*}
 \kappa^4=\|\tilde{\phi} \|_{L^{\infty}(\{\frac{1}{2}\leq |x|\leq \frac{3}{2}\} )}^2
& \les \|\widetilde{\phi}\|_{H^1(\R^2)}\|(\widetilde{\phi},\nabb \widetilde{\phi})\|_{L^2(\R^2)}\left(1+\ln\frac{\|\widetilde{\phi}\|_{H^{2}(\R^2)}} {\|(\widetilde{\phi},\nabb \widetilde{\phi})\|_{L^2(\R^2)}}\right) \\
&\les (\kappa +\sqrt{\ln(1+t)} (1+t)^{\frac{1-p}{8}})^2 (1+t)^{\frac{1-p}{4}} (1+\ln(1+t)+\ln(1+\kappa )),
\end{align*}
from which we can demonstrate that 
\[
\kappa \les \sqrt{\ln(1+t)} (1+t)^{\frac{1-p}{8}}. 
\]
We thus finished the proof for the  Proposition.
\end{proof}
 \begin{remark}
Away from the light cone the solution decays faster
\begin{align*}
|\phi|\leq C 
(1+t)^{-\frac{p-1}{4}+\epsilon}, \quad |t-r|\geq \frac{t}{2}
\end{align*}
for any positive constant $\epsilon>0$ and some constant $C_{\epsilon}$ also depending on $\epsilon$. To see this, it suffices to repeat the above procedure based on the decay estimate of Proposition \ref{prop:ptdecay:allp}.  
\end{remark}
 
\section{Improved pointwise decay estimates for larger $p$}
In the previous section, we have obtained rough pointwise decay estimates for solution to the Chern-Simons-Higgs equation \eqref{eq:CSH:p:2d} for sufficiently localized initial data for all $p>1$. This section is devoted to improving the decay rate for larger power of $p>4$. In view of Proposition \ref{prop:td:2dCSH}, the potential energy decay rate  for the case when $p>5$ is the same with that for $p=5$. It suffices to restrict our attention to the range $4<p\leq 5$.  The idea in \cite{yang:NLW:2D} for the flat case with vanishing connection field $A$ relied on the fundamental solution to the linear wave equation and the nonlinearity can be controlled in terms of the potential energy decay. For the geometric nonlinear equation studied here, there are extra complicated nonlinear terms when applying the  Poisson paramatrix as we have already seen for the growth of higher order energy.

 Alternatively, we make use of the conformal invariance structure of the system \eqref{eq:CSH:p:2d} to obtain sharp time decay estimate for the solution. Inspired by the work on the long time dynamics for the Maxwell-Klein-Gordon equation in \cite{YangYu:MKG:smooth}, one can first study the solution outside a forward light cone $\{t\leq |x|+R\}$. Since the initial data are sufficiently localized, the initial data on $\{t=0, |x|\geq R\}$ can be sufficiently small by taking $R$ large enough. In particular, the problem in this region is reduced to small data global solution for semilinear wave equation with null condition in space dimension two which have been well studied, see for example \cite{Alinhac01:null:2D:2}, \cite{Alinhac01:null:2D:1}. For simplicity, we only study the solution in the interior region inside the forward light cone $\{t+R\geq |x|\}$ and without loss of generality under the assumption that the initial data are compactly supported in $\{|x|\leq 1\}$.

  In the following, we will also allow the implicit constant in $A\les B$ rely on the initial conformal energy $ \mathcal{E}_{0, 2}$, the first order initial energy $\mathcal{E}_{1, 0}$ for the Cauchy problem to the original system \eqref{eq:CSH:p:2d} and a small positive constant $0<\ep<100^{-1}(p-4)$ which may be different from place to place.

\subsection{Conformal transformation}
In view of finite speed of propagation, the solution is support in $\{t+2\geq |x|\}$. 
Define the hyperboloid
\begin{equation*}
\mathbb{H}:=\left\{(t, x)|(t^*)^2-|x|^2=  t^*\right\}, \quad t^*=t+2
\end{equation*}
and the map
\begin{align*}
\mathbf{\Phi}:(t, x)\longmapsto \left(-\frac{t^*}{(t^*)^2-|x|^2}+1,\quad \frac{x}{(t^*)^2-|x|^2}\right)
\end{align*}
from Minkowski space $\R^{1+2}$  to itself. We are in particular interested in the region inclosed by the hyperboloid $\mathbb{H}$ 
\begin{align*}
\mathbf{D}:=\left\{(t, x)|(t^*)^2-|x|^2\geq   t^*\right\}.
\end{align*}
The image of $\mathbf{D}$ under the above map $\mathbf{\Phi}$ is a backward light cone  
\begin{align*}
\mathbf{\Phi}(\mathbf{D})=\left\{(\tilde{t}, \tilde{x})|\quad \tilde{t}+|\tilde{x}|<1,\quad \tilde{t}\geq 0\right\}.
\end{align*}
For solution $(\phi, A)$ to the Chern-Simons-Higgs equation \eqref{eq:CSH:p:2d}, we want to find an equivalent field equation on $(\mathbf{\Phi}(\mathbf{D}), \tilde{g}=d(\mathbf{\Phi}^{-1})^* m)$. 
Let 
\begin{equation*}
\Lambda(t, x)=(t^*)^2-|x|^2, \quad \tilde{\phi}(\tilde{t}, \tilde{x})=\Lambda^{\frac{1}{2}}\phi(\mathbf{\Phi}^{-1}(\tilde{t}, \tilde{x})),\quad \tilde{A}(\tilde{t}, \tilde{x})=A(\mathbf{\Phi}^{-1}(\tilde{t}, \tilde{x})).
\end{equation*}
Here we emphasize that as a 1-form $A$ is invariant under the change of coordinates.  
 We need to find the equations for $\tilde{\phi}$, $\tilde{A}$ on $\mathbf{\Phi}(\mathbf{D})$. In $(t, x)$ coordinates, note that 
\[
\tilde{g}=\Lambda^{-2}m.
\]
Hence the Christoffel symbols for the metric $\tilde{g}$ under the coordinates system $(t, x)$ are
\begin{align*}
  \tilde{\Gamma}_{\mu\nu}^l 
  &=-\Lambda^{-1}(\delta_\mu^{\ l}\pa_\nu \Lambda+\delta_\nu^{\ l}\pa_\mu \Lambda-m_{\mu\nu}\pa^l \Lambda).
\end{align*}
For the covariant derivative, still under the coordinates $(t, x)$ on $\mathbf{\Phi}(\mathbf{D})$, we compute that 
\begin{align*}
\tilde{D}_{\mu}\tilde{\phi}=\pa_{\mu}(\Lambda^{\frac{1}{2}}\phi)+i A_{\mu} \Lambda^{\frac{1}{2}}\phi= \Lambda^{\frac{1}{2}} D_{\mu}\phi+ \pa_\mu \Lambda^{\frac{1}{2}} \phi.
\end{align*} 
The covariant wave operator verifies  
\begin{align*}
  \Box_{\tilde{A}}\tilde{\phi} & =  \Lambda^2 m^{\mu\nu} \tilde{D}_\mu \tilde{D}_{\nu}\tilde{\phi} \\
&=\Lambda^2 m^{\mu\nu}  \Lambda^{\frac{1}{2}} D_\mu (\Lambda^{-\frac{1}{2}}\tilde{D}_{\mu}\tilde{\phi} )+\Lambda^2 m^{\mu\nu} \pa_\mu \Lambda^{\frac{1}{2}}  \Lambda^{-\frac{1}{2}}\tilde{D}_{\mu}\tilde{\phi}-\Lambda^2 m^{\mu\nu}\tilde{\Gamma}_{\mu\nu}^l \tilde{D}_{l}\tilde{\phi} \\
&=\Lambda^2 m^{\mu\nu}  \Lambda^{\frac{1}{2}} D_\nu ( D_{\mu}\phi+ \Lambda^{-\frac{1}{2}}\pa_\mu \Lambda^{\frac{1}{2}} \phi)-\frac{1}{2}\Lambda  \pa^\mu \Lambda ( \Lambda^{\frac{1}{2}} D_{\mu}\phi+ \pa_\mu \Lambda^{\frac{1}{2}} \phi) \\
  &= \Lambda^{\frac{5}{2}}\Box_A\phi+(\Lambda^2\pa^\mu \Lambda^{\frac{1}{2}}-\frac{1}{2}\Lambda  \pa^\mu \Lambda \Lambda^{\frac{1}{2}} ) D_\mu\phi +(\Lambda^{\frac{5}{2}}\pa^\mu (\Lambda^{-\frac{1}{2}}\pa_\mu \Lambda^{\frac{1}{2}} ) - \frac{1}{2}\Lambda  \pa^\mu \Lambda  \pa_\mu \Lambda^{\frac{1}{2}})\phi \\ 
  &= \Lambda^{\frac{5}{2}}\Box_A\phi=\Lambda^{\frac{5}{2}} |\phi|^{p-1}\phi =\Lambda^{\frac{5-p}{2}}|\tilde{\phi}|^{p-1}\tilde{\phi }.
\end{align*}
Here lowering or raising the indices is with respect to the flat Minkowski metric $m_{\mu\nu}$.
Next we want to find the equation for $\tilde{F}=d\tilde{A}$. Since as a $1$-form $\tilde{A}=A$, in particular we have  
\begin{align*}
\tilde{F}_{\mu\nu}d\tilde{x}^{\mu}\wedge d\tilde{x}^{\nu}=F_{\mu\nu }dx^\mu \wedge dx^\nu.
\end{align*}
As a $1$-form, we can show that  
\begin{align*}
 \tilde{J}_\mu d\tilde{x}^\mu=\Im(\tilde{\phi}\cdot \overline{\tilde{D}_{\mu}\tilde{\phi}})d\tilde{x}^{\mu}=\Lambda \Im (\phi\cdot \overline{D_{\mu}\phi})dx^\mu=\Lambda J_\mu dx^\mu.
\end{align*}
On $\mathbf{D}$, the Maxwell field $F$ verifies the equation $F_{\mu\nu}= \epsilon_{\mu\nu\ga} J^{\gamma}$ with $\epsilon_{\mu\nu\ga}$ the totally skew symmetric tensor such that $\epsilon_{012}=1$.  
For fixed $\mu \neq \nu$ and under coordinates $(t, x)$, we then can show that 
\begin{align*}
\tilde{F}_{\mu\nu} \epsilon_{\mu\nu\ga }d\tilde{x}^\mu \wedge d\tilde{x}^\nu \wedge d\tilde{x}^{\ga} &=\tilde{F}_{\mu\nu} d\tilde{t}\wedge d\tilde{x}^1\wedge d\tilde{x}^2=\tilde{F}_{\mu\nu}\Lambda^{-3}dt\wedge dx^1\wedge dx^2\\
&=F_{\alpha \beta} dx^\alpha \wedge dx^\beta \wedge d\tilde{x}^\gamma \epsilon_{\mu\nu\ga }\\
&=\epsilon_{\alpha \beta l} J^l dx^\alpha \wedge dx^\beta \wedge dx^l \frac{\pa \tilde{x}^\gamma}{\pa x^l}\epsilon_{\mu\nu\ga }\\
& = m^{ll} \Lambda^{-1} \tilde{J}_{k}\frac{\pa \tilde{x}^k}{\pa x^l}  \frac{\pa \tilde{x}^\gamma}{\pa x^l}\epsilon_{\mu\nu\ga } dt \wedge dx^1 \wedge dx^2,
\end{align*}
from which we conclude that 
\begin{align*}
 \tilde{F}_{\mu\nu} 
 =\epsilon_{\mu\nu\ga }  m^{ll} \Lambda^{2} \tilde{J}_{k}\frac{\pa \tilde{x}^k}{\pa x^l}  \frac{\pa \tilde{x}^\gamma}{\pa x^l}= \epsilon_{\mu\nu\ga } \tilde{J}_{k} \tilde{g}(d\tilde{x}^k, d\tilde{x}^\ga)  =\epsilon_{\mu\nu\ga}\tilde{J}^{\gamma}.
\end{align*}
These computations imply that the Chern-Simons equation \eqref{eq:CSH:p:2d} for $(\phi, A)$ on $\mathbf{D}\subset\R^{1+2}$ is  
equivalent to 
\begin{equation}
\label{eq:CS:cf}
 \Box_{\tilde{A}}\tilde{\phi} =\Lambda^{\frac{5-p}{2}} |\tilde\phi|^{p-1}\tilde{\phi },\quad \tilde{F}_{\mu\nu}=\epsilon_{\mu\nu\ga} J[\tilde{\phi}]^{\ga}.
 \end{equation} 
on $\mathbf{\Phi}(\mathbf{D})$ under the coordinates $(\tilde{t}, \tilde{x})$ with 
\[
\Lambda=(t^*)^2-|x|^2=(\tilde{t}^2-|\tilde{x}|^2)^{-1}. 
\]
Moreover under the coordinates $(\tilde{t}, \tilde{x})$ on $\mathbf{\Phi}(\mathbf{D})$, note that 
\begin{align*}
2+t+r=(1-\tilde{t}-\tilde{r})^{-1},\quad 2+t-r=(1-\tilde{t}+\tilde{r})^{-1},\quad dtdx=\Lambda^{3}d\tilde{t}d\tilde{x}. 
\end{align*}
Thus in terms of $\tilde{\phi}$ on $\mathbf{\Phi}(\mathbf{D})$, we have the bound 
\begin{equation}
\label{eq:bd4:po}
\iint_{\tilde{t}+|\tilde{x}|\leq 1} (1-\tilde{t}-|\tilde{x}|)^{-\gamma} \Lambda^{\frac{5-p}{2}}|\tilde{\phi}|^{p+1}d\tilde{t}d\tilde{x}  
= \iint_{\mathbf{D}}   (2+t+r)^{\gamma }  |\phi|^{p+1}dtdx  \les 1
\end{equation}
 for all $\gamma <\frac{p-3}{2}$ in view of the potential energy decay estimate of Proposition \ref{prop:td:2dCSH}. 
With this spacetime bound, we establish a type of energy estimate for solution to \eqref{eq:CS:cf}. 
We start with the energy identity
\begin{align*}
\tilde{\pa}^\mu ( T[\tilde{\phi}]_{\mu\nu}  X^\nu )
&=T[\tilde{\phi}]^{\mu\nu} \pi_{\mu\nu}^X-\frac{1}{p+1}X(\Lambda^{\frac{5-p}{2}})|\tilde{\phi}|^{p+1}
\end{align*}
for solution $(\tilde{A},\tilde{\phi})$ to \eqref{eq:CS:cf} with 
\[
T[\tilde{\phi}]_{\mu\nu}=\l  \tilde{D}_\mu \tilde{\phi}, \tilde{D}_\nu \tilde{\phi}\r  -\frac{1}{2}\tilde{m}_{\mu \nu} (\l \tilde{D}^{\ga}\tilde{\phi}, \tilde{D}_\ga\tilde{\phi} \r+\frac{2}{p+1}\Lambda^{\frac{5-p}{2}}|\tilde{\phi }|^{p+1})
\]
and $\tilde{m}$ the Minkowski metric on $\R^{1+2}$. For $\gamma < \frac{p-3}{2}\leq 1$, let 
\[
X=(1-\tilde{t}-\tilde{r})^{1-\gamma}\pa_{\tilde{t}}=(1-\tilde{t}-\tilde{r})^{1-\gamma }\tilde{\pa}_{0}.
\]
We compute that 
\begin{align*}
&T[\tilde{\phi}]_{0\nu}  X^\nu=T[\tilde{\phi}]_{00} (1-\tilde{t}-\tilde{r})^{1-\gamma}=\frac{1}{2}(1-\tilde{t}-\tilde{r})^{1-\gamma} (|\tilde{D}\tilde{\phi}|^2+\frac{2\Lambda^{\frac{5-p}{2}}}{p+1}|\tilde{\phi }|^{p+1}),\\
&T[\tilde{\phi}]^{\mu\nu} \pi_{\mu\nu}^X=(1-\gamma)(1-\tilde{t}-\tilde{r})^{-\gamma} (T[\tilde{\phi}]_{00}-T[\tilde{\phi}]_{0\tilde{r}})\geq 0, \\ 
& X(\Lambda^{\frac{5-p}{2}}) = (p-5)(1-\tilde{t}-\tilde{r})^{-\gamma }\Lambda^{\frac{5-p}{2}} (1-\tilde{t})(1-\tilde{t}+\tilde{r})^{-1}.
\end{align*}
Applying Stokes formula to the region $\{(s, \tilde{x})|s+|\tilde{x}|\leq 1, 0\leq s\leq \tilde{t}\}$  leads to the energy estimate 
\begin{equation}
\label{eq:WE:2dcf}
\begin{split}
&\int_{|\tilde{x}|\leq 1-\tilde{t}} (1-\tilde{t}-\tilde{r})^{1-\gamma} (|\tilde{D}\tilde{\phi}|^2+\frac{2}{p+1}\Lambda^{\frac{5-p}{2}}|\tilde{\phi}|^{p+1})  d\tilde{x} \\
 &\leq \int_{\tilde{t}=0, |\tilde{x}|\leq 1 } (1 -\tilde{r})^{1-\gamma} (|\tilde{D}\tilde{\phi}|^2+\frac{2}{p+1}\Lambda^{\frac{5-p}{2}}|\tilde{\phi}|^{p+1} ) d\tilde{x}\\ 
&\quad +2|p-5|\int_{0}^{\tilde{t}}\int_{|\tilde{x}|\leq 1-s}  (1-\tilde{t}-\tilde{r})^{-\gamma}\Lambda^{\frac{5-p}{2}}   |\tilde{\phi}|^{p+1}d\tilde{x}ds\\
&\les 1
\end{split}
\end{equation}
in view of the bound \eqref{eq:bd4:po}. Here the first term on the right hand side is the weighted initial energy on $\{(0, \tilde{x})||\tilde{x}|\leq 1\}$.

Finally the rough pointwise decay estimate of the main theorem \ref{thm:main:pd:gen} indicates that 
\begin{equation}
\label{eq:pt:phi:cf:2d}
\begin{split}
|\tilde{\phi}(\tilde{t}, \tilde{x})|&=|\Lambda^{\frac{1}{2}}\phi(t, x)|\les (2+t+r)^{-\frac{p-1 }{8}} (2+t+r)^{\frac{1}{2}}(2+t-r)^{\frac{1}{2}}\sqrt{\ln(2+t+r)} \\ 
&\les (1-\tilde{t}-\tilde{r})^{\frac{p-5-\ep_1}{8} }(1-\tilde{t})^{-\frac{1}{2}}
\end{split}
\end{equation}
for all $\ep_1>0$. 

 \subsection{Singular Chern-Simons-Higgs equation on compact region}
 The previous discussion indicates that the long time dynamic behaviors for solution to the Chern-Simons-Higgs equation \eqref{eq:CSH:p:2d} is then reduced to study the singular equation \eqref{eq:CS:cf} on the compact region $\mathbf{\Phi}(\mathbf{D})$. To make use of the geometric paramatrix, instead of investigating the system \eqref{eq:CS:cf} in space dimension two, we lift it to space dimension three. 

 Let $B$ be the ball with radius $1$ in $\mathbb{R}^3$, that is, 
\[
B=\{|x|=|(x_1, x_2, x_3)|\leq 1\}.
\]
 The maximal Cauchy development is the backward cone  $\mathcal{J}^+(B)$
\begin{align*}
\mathcal{J}^+(B)=\{t+|x|\leq 1\}.
\end{align*}
Let $r=|x|$ and define the null coordinates 
\[
u=\frac{t-r}{2},\quad v=\frac{t+r}{2},\quad u_*=1-t+r,\quad v_*=1-t-r.
\]
At any fixed point, we can choose a null frame $\{L=\pa_t+\pa_r, \Lb=\pa_t-\pa_r, e_1, e_2\}$ such that 
\begin{equation*}
 \begin{split}
&\nabla_{e_j} L=r^{-1}e_j, \quad \nabla_{e_j}\Lb=-r^{-1}e_j, \quad \nabla_{e_1} e_2=\nabla_{e_2}e_1=0, \quad \nabla_{e_j}e_j=-r^{-1}\pa_r.
 \end{split}
\end{equation*}
Here $\nabla$ is the covariant derivative in $\R^3$. 

\def\J{\mathcal{J}^{-}}
\def\N{\mathcal{N}^{-}}
For any point $q=(t_0, x_0)\in \mathcal{J}^{+}(B)$, let $\N(q)$ be the past null cone with vertex $q$, that is, 
\begin{align*}
   \mathcal{N}^{-}(q):=\big\{(t, x)\big| t_0-t=|x-x_0|,\ t\geq 0\big\}
 \end{align*}
 and $\mathcal{J}^{-}(q)$ to be the past of the point $q$, i.e., the region bounded by $\mathcal{N}^{-}(q)$ and the initial hypersurface $B$. For $r>0$ and $q=(t_0, x_0)\in\mathbb{R}^{1+3}$, let  $B_q(r)$ and $S_q(r)$ be the  spatial ball and sphere at time $t_0$ with radius $r$ centered at the point $q$, that is,
\begin{align*}
  B_q( r)=\big\{(t,x)\big| t=t_0, |x-x_0|\leq r\big\},\quad S_q(r)=\big\{(t,x)\big| t=t_0, |x-x_0|= r\big\}.
\end{align*}
Unless it is specified, we will always  fix a point $q=(t_0, x_0)\in \mathcal{J}^{+}(B)$ during the argument and we will also use the translated coordinate system $(\widetilde{t}, \widetilde{x})$ so that it is  centered at the point $q$. More specifically,
\[
\widetilde{t}=t-t_0,\quad \widetilde{x}=x-x_0,\quad \widetilde{r}=|\widetilde{x}|,\quad \widetilde{\om}=\frac{\widetilde{x}}{|\widetilde{x}|},\quad \widetilde{u}=\f12 (\widetilde{t}-\widetilde{r}),\quad \widetilde{v}=\f12(\widetilde{t}+\widetilde{r}).
\]
In particular, we have the associated null frame $\{\widetilde{L}, \widetilde{\Lb}, \widetilde{e}_1, \widetilde{e}_2\}$.

Now we lift the equation \eqref{eq:CS:cf} on $\R^{1+2}$ to higher dimensional space $\R^{1+3}$ by simply extending the $1$-form $\tilde{A}$ and scalar field $\tilde{\phi}$ to $(A, \phi)$ defined on  $\R^{1+3}$ such that 
\[
A(t, x)=\tilde{A}(t, x_1, x_2),\quad A_3=0,\quad \phi(t, x)=\tilde{\phi}(t, x_1, x_2).
\] 
Here we emphasize that this $(A, \phi)$ is different from that as solution to the original system \eqref{eq:CSH:p:2d}. We abuse the notation for avoiding too many notations. 
Then the covariant derivative $D_\mu$ can also be extended to $\R^{1+3}$ associated to the connection field $A$. The commutator gives the Maxwell field $F$
\[
F=dA,\quad i F_{\mu\nu}=[D_\mu, D_\nu],\quad \mu, \nu=0, 1, 2, 3. 
\]
We still use $m=diag\{-1, 1, 1, 1\}$ to denote the Minkowski metric on $\R^{1+3}$. 
Since $A_3=0$ and $A$ is independent of the variable $x_3$, it is obvious that 
\[
F_{3\mu}=0,\quad \mu=0, 1, 2, 3.  
\]
Therefore equation \eqref{eq:CS:cf} is equivalent to 
\begin{equation}
\label{eq:CSH:3d}
\begin{cases}
 \Box_{A}\phi =\Lambda^{\frac{5-p}{2}} |\phi|^{p-1} \phi, \quad \Lambda=((1-t)^2-x_1^2- x_2^2 )^{-1} ,\\ 
 F_{3\ga}=0, \quad J[\phi]_3=0,\quad \ga=0, 1, 2, \\ 
 F_{\mu\nu}=\epsilon_{\mu\nu\ga} J[\phi]^{\ga},\quad J[\phi]_{\ga}=\Im(\phi\cdot\overline{D_{\ga}\phi}),\quad  \mu,\nu=0, 1, 2. 
 \end{cases}
 \end{equation} 
 Here $\epsilon_{\mu\nu\ga}$ is the skew symmetric tensor such that $\epsilon_{012}=1$. By our assumption that the initial data for the original system \eqref{eq:CSH:p:2d} in $\R^{1+2}$ are compactly supported, the initial data for \eqref{eq:CSH:3d} are also compactly supported in variable $x_1$ and $x_2$ and is independent of the third variable $x_3$. 
Moreover the solution $\tilde{\phi}$ in the previous section verifies the a priori bounds \eqref{eq:bd4:po}, \eqref{eq:WE:2dcf},\eqref{eq:pt:phi:cf:2d} in space dimension two.
As solution in space dimension three here, these estimates imply the weighted potential energy bound
\begin{equation} 
\label{eq:bd4:po:3}
\begin{split}
  &\iint_{\mathcal{J}^+(B)} (1-t)^{-\frac{1}{2}}(1-t-r_1)^{-\gamma-\frac{1}{2}}\Lambda^{\frac{5-p}{2}} |\phi|^{p+1}dtdx_1dx_2 dx_3 \\ 
&=\int_{t+r_1\leq 1} \Lambda^{\frac{5-p}{2}} (1-t)^{-\frac{1}{2}}(1-t-r_1)^{-\gamma-\frac{1}{2}} |\tilde{\phi}|^{p+1} dt dx_1 dx_2 \int_{|x_3|^2\leq (1-t)^2-r_1^2}   dx_3 \\ 
& \leq 2\sqrt{2}  \int_{t+r_1\leq 1} (1-t-r_1)^{-\gamma}\Lambda^{\frac{5-p}{2}} |\tilde\phi|^{p+1} dt dx_1 dx_2 \\ 
&\les 1
\end{split}
\end{equation}
and the energy decay estimate 
\begin{equation}
\label{eq:WEE:3d}
\begin{split}
&\int_{|x|\leq 1- t} (1-t-r_1)^{\frac{1}{2}-\gamma}(|D \phi |^2 +\frac{2}{p+1}\Lambda^{\frac{5-p}{2}}|\phi|^{p+1}) dx_1 dx_2 dx_3 \\ 
&=
\int_{|(x_1, x_2)|\leq 1- t} (1-t+r_1)^{\frac{1}{2}}(1-t-r_1)^{1-\gamma }|D\tilde{\phi}|^2 dx_1 dx_2\les \sqrt{1-t}
\end{split}
\end{equation}
as well as the pointwise bound 
 \begin{equation}
\label{eq:pt:phi:cf:3d}
\begin{split}
|\phi( t, x_1, x_2, x_3)|=|\tilde{\phi}(t, x_1, x_2)| \les (1- t- r_1)^{\frac{p-5-\ep_1}{8}}(1- t)^{-\frac{1}{2}}.
\end{split}
\end{equation}
Here $r_1=\sqrt{x_1^2+x_2^2}$ and note that $\Lambda$ and $\phi$ is independent of the variable $x_3$.

Our key estimate is a uniform weighted energy flux bound through the backward light cone $\mathcal{N}^{-}(q)$. 
\begin{Prop}
\label{prop:EF:cone:gamma}
Let $\gamma\in [0, \frac{1}{2})$. For any point $q=(t_0, x_0)\in \mathcal{J}^{+}(B)$,  
we have the   uniform energy flux bound
\begin{align*}
  \int_{\mathcal{N}^{-}(q)} (1-t)^{-\gamma}\Big( |\widetilde{\D} \phi|^2  + |D_{\widetilde{\Lb}} \phi|^2 + \Lambda^{\frac{5-p}{2}}    |\phi |^{p+1}\Big)  \widetilde{r}^2d\widetilde{u}d\widetilde{\om}   \les 1 
\end{align*} 
and the local decay of the energy flux 
\begin{align*}
  \int_{ \mathcal{N}^{-}(q) \cap\{  t_0-\tilde{r}\leq t\leq t_0\} }  \Big( |\widetilde{\D} \phi|^2  + |D_{\widetilde{\Lb}} \phi|^2 + \Lambda^{\frac{5-p}{2}}    |\phi |^{p+1} \Big)  s^2ds d\widetilde{\om}   \les \sqrt{\tilde{r}} ,\quad 0\leq \tilde{r}\leq t_0.
\end{align*} 
Here recall that the null frame $\{\tilde{L}, \tilde{e_1}, \tilde{e}_2\}$ is with respect to the coordinate  $(\widetilde{t}, \widetilde{x})$ centered at the point $q=(t_0, x_0)$.
\end{Prop}
\begin{proof}
The energy momentum tensor for the scalar field $\phi$ defined on $\R^{1+3}$ as solution to \eqref{eq:CSH:3d} is of the same form
\begin{align*}
T[\phi]_{\mu\nu}=\l  D_\mu \phi, D_\nu\phi\r  -\frac{1}{2}m_{\mu\nu} (\l D^{\ga}\phi, D_\ga\phi \r+\frac{2}{p+1}\Lambda^{\frac{5-p}{2}}|\phi|^{p+1}),
\end{align*}
in which $m_{\mu\nu}$ is the flat metric on $\R^{1+3}$. We also have the energy identity for solution $(A, \phi)$ to \eqref{eq:CSH:3d}
\begin{align*}
\pa^\mu ( T[ \phi]_{\mu\nu}  X^\nu )
&=T[\phi]^{\mu\nu} \pi_{\mu\nu}^X-\frac{1}{p+1}X(\Lambda^{\frac{5-p}{2}})| \phi|^{p+1}
\end{align*}
for any vector field $X$ defined on $\mathcal{J}^{+}(B)$. 
Now take the vector field 
\[
X=(1-t)^{-\gamma}\pa_t. 
\]
We compute that 
\begin{align*}
&T[\phi]^{\mu\nu} \pi_{\mu\nu}^X=-\frac{1}{2}\gamma (1- t)^{-1-\gamma} (|D \phi|^2+\frac{2\Lambda^{\frac{5-p}{2}}}{p+1}| \phi |^{p+1}), \\ 
& X(\Lambda^{\frac{5-p}{2}}) = (5-p)(1-t)^{-\gamma}\Lambda^{\frac{5-p}{2}} (1-t-r_1)^{-1} (1- t)(1- t+ r_1)^{-1}.
\end{align*} 
Here $\pi^X_{\mu\nu}=\frac{1}{2}(\pa_\mu X_\nu+\pa_\nu X_\mu)$ is the deformation tensor for the metric $m$ along the vector field $X$ in $\R^{1+3}$. In view of the estimate \eqref{eq:bd4:po:3} for the weighted potential energy and the weighted energy estimate \eqref{eq:WEE:3d}, we can bound that 
\begin{align*}
 &\iint_{\mathcal{J}^+(B)}|T[\phi]^{\mu\nu} \pi_{\mu\nu}^X-\frac{1}{p+1}X(\Lambda^{\frac{5-p}{2}})| \phi|^{p+1}|dxdt \\
&\les \int_0^1 (1- t)^{-1-\gamma} \int_{|x|\leq 1-t}   |D \phi|^2+\frac{2\Lambda^{\frac{5-p}{2}}}{p+1}| \phi |^{p+1} dxdt \\ 
&\quad  + \iint_{\mathcal{J}^+(B)}  (1-t)^{-\gamma}\Lambda^{\frac{5-p}{2}} (1-t-r_1)^{-1} (1- t)(1- t+ r_1)^{-1} |\phi |^{p+1} dxdt \\ 
&\les \int_0^1 (1-t)^{-1-\gamma}\sqrt{1-t} dt+   \iint_{\mathcal{J}^+(B)}  (1-t)^{-\frac{1}{2}}\Lambda^{\frac{5-p}{2}} (1-t-r_1)^{-1}  |\phi |^{p+1} dxdt \\ 
&\les 1
\end{align*}
for all $\gamma <\frac{1}{2}$. Then for the region $\mathcal{J}^{-}(q)$, Stokes formula leads to  the energy estimate 
\begin{align*}
&\int_{\mathcal{N}^{-}(q)} (1-t)^{-\gamma}(|D_{\tilde{\Lb}}\phi|^2+|\tilde{\D}\phi|^2+\frac{2}{p+1}\Lambda^{\frac{5-p}{2}}|\phi|^{p+1}) \widetilde{r}^2d\widetilde{u}d\widetilde{\om} \\ 
&\leq  |\iint_{\mathcal{J}^{-}(q)} \pa^\mu ( T[ \phi]_{\mu\nu}  X^\nu ) d\vol| +\int_{\mathcal{J}^{-}(q)\cap B} |D\phi(0, x)|^2+\frac{2}{p+1}\Lambda^{\frac{5-p}{2}}|\phi |^{p+1} dx \les 1.
\end{align*}  
This gives the uniform weighted energy flux bound through the backward light cone $\mathcal{N}^{-}(q)$. In particular for the case when $\gamma=0$, standard energy estimate applied to the region $\mathcal{J}^{-}(q)\cap \{t_0-\tilde{r}\leq t\leq t_0\}$ shows that 
\begin{align*}
&\int_{ \mathcal{N}^{-}(q) \cap\{  t_0-\tilde{r}\leq t\leq t_0\} }  \Big( |\widetilde{\D} \phi|^2  + |D_{\widetilde{\Lb}} \phi|^2 + \Lambda^{\frac{5-p}{2}}    |\phi |^{p+1} \Big)  s^2ds d\widetilde{\om} \\ 
&\les \int_{ B(t_0-\tilde{r}, x_0)(\tilde{r})}   |D \phi|^2  + \Lambda^{\frac{5-p}{2}}    |\phi |^{p+1} dx +\iint_{ \mathcal{J}^{-}(q)\cap \{t_0-\tilde{r}\leq t\leq t_0\}}  \Lambda^{\frac{5-p}{2}} (1-t-r_1)^{-1} |\phi |^{p+1} dxdt. 
\end{align*}
Since $\phi$ is independent of the third variable $x_3$, in view of the energy estimate \eqref{eq:WE:2dcf} for $\tilde{\phi}$ in space dimension two with $\gamma =\frac{1}{2}<\frac{p-3}{2}$, we can estimate that 
 \begin{align*}
 \int_{ B(t_0-\tilde{r}, x_0)(\tilde{r})}   |D \phi|^2  + \Lambda^{\frac{5-p}{2}}    |\phi |^{p+1} dx & = \int_{\tilde{r}_1 \leq \tilde{r}}  (|D \tilde{\phi} |^2  + \Lambda^{\frac{5-p}{2}}    |\tilde{\phi}  |^{p+1} ) \sqrt{\tilde{r}^2 -\tilde{r}_1^2}dx_1 dx_2 \\ 
 &\les \sqrt{\tilde{r}}\int_{\tilde{r}_1 \leq \tilde{r}}  (|D \tilde{\phi} |^2  + \Lambda^{\frac{5-p}{2}}    |\tilde{\phi}  |^{p+1} ) \sqrt{ 1-(t_0-\tilde{r})-r_1 }dx_1 dx_2\\
 &\les \sqrt{\tilde{r}}.
\end{align*}
Here $\tilde{r}_1=\sqrt{\tilde{x}_1^2+\tilde{x}_2^2}$ and we have used the inequality 
\[
\tilde{r}-\tilde{r}_1\leq \tilde{r}-(r_{01}+r_1)\leq 1-t_0+\tilde{r}-r_1,
\]
in which $r_{01}=\sqrt{x_{01}^2+x_{02}^2}$, $x_0=(x_{01}, x_{02}, x_{03})$. 

Similarly by using the weighted potential energy bound \eqref{eq:bd4:po} for $\tilde{\phi}$ in space dimension two with $\gamma =\frac{1}{2}$, we also have 
 \begin{align*}
 & \iint_{ \mathcal{J}^{-}(q)\cap \{t_0-\tilde{r}\leq t\leq t_0\}}  \Lambda^{\frac{5-p}{2}} (1-t-r_1)^{-1} |\phi |^{p+1} dxdt \\ 
 &=\int_0^{\tilde{r}}\int_{ \tilde{r}_1\leq s} \sqrt{s^2-\tilde{r}_1}\Lambda^{\frac{5-p}{2}} (1-t-r_1)^{-1} |\tilde{\phi} |^{p+1} dx_1 dx_2 ds \\ 
 &\les  \sqrt{\tilde{r}}\int_0^{\tilde{r}}\int_{ \tilde{r}_1\leq s}  \Lambda^{\frac{5-p}{2}} (1-t-r_1)^{-\frac{1}{2}} |\tilde{\phi} |^{p+1} dx_1 dx_2 ds\\
 &\les \sqrt{\tilde{r}}.
\end{align*}
These two estimates are sufficient to conclude the Proposition. 
\end{proof}

\subsection{Geometric Kirchoff-Sobolev paramatrix and sharp decay estimates}

Let's recall the Kirchoff-Sobolev paramatrix introduced by Klainerman and Rodnianski \cite{Kl:paramatrix} in $\mathbb{R}^{1+3}$.   Under the coordinates $(\tilde{t}, \tilde{x})$ together with the null frame $\{\tilde{L}, \tilde{\Lb}, \tilde{e}_1, \tilde{e}_2\}$ centered at the point $q$,  let $h$ be  the function on the backward cone $\mathcal{N}^{-}(q)=\{(\tilde{t}, \tilde{x})| \tilde{v}=0\}$ such that 
\begin{align*}
D_{\tilde{\Lb}} h=0, \quad h(q)= |\phi(q)|^{-1}\phi(q).
\end{align*}
Here without loss of generality assume that $\phi(q)=\phi(t_0, x_0)\neq 0$. 
The paramatrix in Theorem 4.1 in \cite{Kl:paramatrix} gives the  representation formula for $\phi$ 
\begin{equation}
\label{eq:rep4phi}
\begin{split}
4\pi|\phi(q) | &=\int_{B} \langle \widetilde{r}^{-1}h\delta(\widetilde{v}), D_0 \phi\rangle-\langle D_0(\widetilde{r}^{-1}h\delta(\widetilde{v})), \phi \rangle dx +\iint_{\mathcal{N}^{-}(q)}\langle\widetilde{\lap}_A h-i\widetilde{\rho}  h, \phi\rangle\widetilde{r} d\widetilde{r}d\widetilde{\om} \\ 
&\qquad \qquad -\iint_{\mathcal{N}^{-}(q)}\langle  h, \Box_A\phi\rangle\widetilde{r} d\widetilde{r}d\widetilde{\om},
\end{split}
\end{equation}
in which $\delta(\widetilde{v})$ is the Dirac distribution and  $\widetilde{\rho}=\f12 F(\widetilde{\Lb}, \widetilde{L})$. 
Note that
\[
|h|\leq 1.
\]
The first term on the right hand side can be bounded by the initial data
\begin{align*}
&\left|\int_{B}\langle\widetilde{r}^{-1}h\delta(\widetilde{v}), D_0 \phi\rangle-\langle D_0(\widetilde{r}^{-1}h\delta(\widetilde{v})), \phi \rangle dx\right| \\ 
&= \left|\int_{B } \langle h\delta(\widetilde{v}), \widetilde{r} D_0 \phi \rangle -\langle D_0(h)\delta(\widetilde{v})+h \delta(\widetilde{v})', \widetilde{r}\phi\rangle d\widetilde{r}d\widetilde{\om}\right|\\
&=\left|\int_{B } \langle h\delta(\widetilde{v}), \widetilde{r} D_0 \phi\rangle - \langle D_0(h)\delta(\widetilde{v}), \widetilde{r}\phi\rangle+\delta(\widetilde{v})(\langle D_{\widetilde{r}}(h), \widetilde{r}\phi\rangle +\langle h, D_{\widetilde{r}}(\widetilde{r}\phi)\rangle) d\widetilde{r}d\widetilde{\om}\right|\\
&\leq \int_{\widetilde{\om}}t_0|D\phi(0, x_0+t_0\widetilde{\om})|+|\phi|(0, x_0+t_0\widetilde{\om}) d\widetilde{\om}\les 1.
\end{align*}
Here by the assumption that the initial data are compactly supported. See precise dependence on the initial weighted energy in \cite{YangYu:MKGrough}.

Before we proceed to estimate the nonlinear terms, we recall the following Lemma.
\begin{Lem}
\label{lem:bd:vga}
Fix a point $(t_0, x_0)\in \mathcal{J}^+(B)$. For all $0\leq \widetilde{r}\leq t_0$, $t=t_0-\widetilde{r}$, $r=|x_0+\widetilde{r}\widetilde{\om}|$ and $0<\ga <1$,
it holds that 
 \begin{equation*}
\int_{S_{(t, x_0)}(\widetilde{r})}(1-t-r)^{-\ga}d\widetilde{\om}\les (1-t)^{\ga}(1-t_0+r_0)^{-\ga} \widetilde{r}^{-\ga}.
\end{equation*}
\end{Lem}
\begin{proof}
See for example Lemma 5.2 in    \cite{YangYu:MKGrough}.
\end{proof}

Now for the pure power nonlinearity, 
denote 
\[
\mathcal{M}(t)=\sup\limits_{|x|\leq 1-t} |\phi(t, x) (1-t+r)^{\frac{5-p+2\ep }{2p_1'}}|^{p_1'}
\]
with $p_1'$ verifying the relation 
\begin{equation}
\label{eq:def4p1}
\frac{p-5-\ep_1}{8} (p+1-2p_1')=-\ep.
\end{equation}
Let $p_1=\frac{p_1'}{p_1'-1}$ be the dual exponent for $p_1'$. In view of the weighted energy flux bound of Proposition \ref{prop:EF:cone:gamma}, we estimate that 
\begin{align*}
  &|\int_{\mathcal{N}^{-}(q)}\langle  h, \Box_A\phi\rangle\widetilde{r}   d\tilde{r}d\tilde{\om}| \leq \int_{\mathcal{N}^{-}(q)}\Lambda^{\frac{5-p}{2}}|\phi|^{p}\ \tilde{r} d\tilde{r}d\tilde{\om}\\
  &\leq \left(\int_{\mathcal{N}^{-}(q)} (1-t)^{-\gamma}\Lambda^{\frac{5-p}{2}}|\phi|^{p+1}\ \tilde{r}^{2} d\tilde{r}d\tilde{\om}\right)^{\frac{1}{p_1}}  \cdot \left(\int_{\mathcal{N}^{-}(q)}(1-t)^{\frac{p_1'}{p_1 }\gamma}\Lambda^{\frac{5-p}{2}}|\phi|^{ p+1-p_1'}\ \tilde{r}^{2-p_1'} d\tilde{r}d\tilde{\om}\right)^{\frac{1}{p_1'}}\\ 
  &\les  \left(\int_{0}^{t_0} \mathcal{M}(t) \int_{S_{(t_0-\tilde{r}, x_0)(\tilde{r})}} (1-t)^{\frac{p_1'}{p_1 }\gamma-\frac{5-p+2\ep }{2}}\Lambda^{\frac{5-p}{2}}|\phi|^{ p+1-2p_1'}\ \tilde{r}^{2-p_1'}   d\tilde{\om} d\tilde{r} \right)^{ \frac{1}{p_1'}}.
\end{align*}
Now we need to estimate the integral on the sphere $S_{(t_0-\tilde{r}, x_0)(\tilde{r})}$ for fixed $\tilde{r}$. Since  $2p_1'<p+1$, by using the rough pointwise decay estimate \eqref{eq:pt:phi:cf:3d}, we have 
\begin{align*}
\Lambda^{\frac{5-p}{2}}|\phi|^{ p+1-2p_1'} & \les (1-t)^{\frac{p-5}{2}}(1-t-r_1)^{\frac{p-5}{2}} (1- t- r_1)^{\frac{p-5-\ep_1}{8} (p+1-2p_1')}(1- t)^{-\frac{p+1-2p_1'}{2}} \\ 
&\les (1-t)^{ p_1' -3} (1-t-r_1)^{\frac{p-5}{2} -\ep}. 
\end{align*}
For sufficiently small $\ep$, it holds that
\[
0<\frac{5-p}{2}+\ep<1.
\]
Then Lemma \ref{lem:bd:vga} implies that 
\begin{align*}
\int_{S_{(t_0-\tilde{r}, x_0)(\tilde{r})}}  (1-t-r_1)^{\frac{p-5}{2}-\ep } d\tilde{\om }\les (1-t)^{\frac{5-p}{2}+\ep } \tilde{r}^{\frac{p-5}{2}- \ep} (1-t_0+r_0)^{\frac{p-5}{2}- \ep}.
\end{align*}
Here note that $r_1=\sqrt{x_1^2+x_2^2}\leq r$. Taking $\gamma =\frac{1}{2}-\ep $, we then have the bound 
\begin{align*}
&\int_{S_{(t_0-\tilde{r}, x_0)(\tilde{r})}} (1-t)^{\frac{p_1'}{p_1 }\gamma-\frac{5-p+2\ep }{2}}\Lambda^{\frac{5-p}{2}}|\phi|^{ p+1-2p_1'}\ \tilde{r}^{2-p_1'}   d\tilde{\om} \\ 
&\les (1-t)^{\frac{5-p}{2}+\ep } \tilde{r}^{\frac{p-5}{2}- \ep} (1-t_0+r_0)^{\frac{p-5}{2}- \ep} (1-t)^{ p_1' -3} (1-t)^{\frac{p_1'}{p_1 }\gamma-\frac{5-p+2\ep }{2}} \ \tilde{r}^{2-p_1'}  \\
&\les  (1-t_0+r_0)^{\frac{p-5}{2}- \ep} \tilde{r}^{-1+\frac{(p-1-\ep_1)\ep}{5+\ep_1-p} }(1-t)^{ \frac{3p-11}{4}-\frac{p-1}{2}\ep -\frac{6\ep}{5+\ep_1-p} } \\
&\les (1-t_0+r_0)^{\frac{p-5}{2}- \ep} \tilde{r}^{-1+\ep } 
\end{align*}
for sufficiently small $\ep>0$ and $\ep_1>0$. Here recall that we assume $4<p\leq 5$. We remark here that for the flat case with vanishing connection field $A=0$ in $\R^{1+2}$, sharp time decay for solution to the original equation \eqref{eq:CSH:p:2d} was shown for $p>\frac{11}{3}$. This is also suggested by the method here. However the proof needs to be modified and will be more involved when $p\leq 4$. We are not pursuing this in this work.

To summarize, the above computations show that 
\begin{align}
\label{eq:spt:pterm}
  (1-t_0+r_0)^{\frac{5+2\ep -p}{2}}\left|\int_{\mathcal{N}^{-}(q)}\langle  h, \Box_A\phi\rangle\widetilde{r}   d\tilde{r}d\tilde{\om}\right|^{p_1'} \les \int_0^{t_0}\mathcal{M}(t_0-\tilde{r}) \tilde{r}^{-1+\ep }d\tilde{r}.
\end{align}
Here by our notation  $t=t_0-\tilde{r}$.

\bigskip 

Next we control the second integral in \eqref{eq:rep4phi} arising from the connection field $A$.  Since the representation formula \eqref{eq:rep4phi} relies   on the angular derivative of $h$  on the cone $\mathcal{N}^{-}(q)$, we commute the transport equation for $h$ with covariant angular derivative. Under the null frame $\{\tilde{L}, \tilde{\Lb}, \tilde{e}_1, \tilde{e}_2\}$ and by the definition of $h$, we have 
\[
\tilde{\Lb}\langle h, h \rangle =2\langle D_{\tilde{\Lb}}h, h\rangle =0.
\]
Since $\langle h(q), h(q)\rangle =1$,  it holds that on the cone $\mathcal{N}^{-}(q)$
\[
\langle h, h \rangle =|h|^2=1.
\]
Taking angular derivative, we have 
\[
\tilde{r}\tilde{e}_j \langle h, h \rangle =2\langle D_{\tilde{r}\tilde{e}_j} h, h \rangle =0. 
\]
To understand $D_{\tilde{r}\tilde{e}_j} h$, note that $[\tilde{\Lb}, \tilde{r}\tilde{e}_j]=0$. Commuting with the transport equation for $h$, we derive
\begin{align*}
D_{\tilde{\Lb}}(D_{\tilde{r}\tilde{e}_j} h)=[D_{\tilde{\Lb}}, D_{\tilde{r}\tilde{e}_j}]h=i F(\tilde{\Lb}, \tilde{r}\tilde{e}_j)h. 
\end{align*}
In particular we have  
\[
\tilde{\Lb}\langle D_{\tilde{r}\tilde{e}_j} h, ih\rangle =  \langle i F(\tilde{\Lb}, \tilde{r}\tilde{e}_j)h , ih\rangle=F(\tilde{\Lb}, \tilde{r}\tilde{e}_j),
\]
from which we derive that 
\begin{align*}
\langle D_{\tilde{r}\tilde{e}_j} h, ih\rangle =-\int_0^{\tilde{r}} F(\tilde{\Lb},  \tilde{e}_j)s ds.
\end{align*}
We then show that 
\begin{align*}
\tilde{\Lb} |D_{\tilde{r}\tilde{e}_j} h|^2= 2 \langle D_{\tilde{\Lb}}D_{\tilde{r}\tilde{e}_j} h, D_{\tilde{r}\tilde{e}_j} h\rangle & = 2 \langle i F(\tilde{\Lb}, \tilde{r}\tilde{e}_j)h , D_{\tilde{r}\tilde{e}_j} h\rangle \\ 
&=-2 F(\tilde{\Lb}, \tilde{r}\tilde{e}_j) \int_0^{\tilde{r}} F(\tilde{\Lb},  \tilde{e}_j) sds\\
&=\tilde{\Lb} \big( \int_0^{\tilde{r}} F(\tilde{\Lb},  \tilde{e}_j) sds\big)^2,
\end{align*} 
which leads to 
\begin{align*}
 |D_{\tilde{r}\tilde{e}_1} h|^2+ |D_{\tilde{r}\tilde{e}_2} h|^2=  \sum\limits_{j,k} \big( \int_0^{\tilde{r}} F(\tilde{\Lb},  \tilde{\Omega}_{jk}) ds\big)^2.
\end{align*} 
Here $\tilde{\Omega}_{jk}$ are the angular momentum vector field 
$$\tilde{\Omega}_{jk}=\tilde{x}_j\pa_{\tilde{x}_k}-\tilde{x}_k\pa_{\tilde{x}_j}=\tilde{x}_j\pa_{k}-\tilde{x}_k\pa_{_j}.$$ 
To compute $\tilde{r}^2\lap_A h $, commute with the angular derivative again. We derive that 
\begin{align*}
D_{\tilde{\Lb}}(\tilde{r}^2\lap_A h)&=D_{\tilde{\Lb}}(D^{\tilde{r}\tilde{e}_j}D_{\tilde{r}\tilde{e}_j} h)=[D_{\tilde{\Lb}}, D^{\tilde{r}\tilde{e}_j}]D_{\tilde{r}\tilde{e}_j}h+D^{\tilde{r}\tilde{e}_j}(iF(\tilde{\Lb}, \tilde{r}\tilde{e}_j)h)\\
&=2iF(\tilde{\Lb}, \tilde{r}\tilde{e}_j)D_{\tilde{r}\tilde{e}_j}h+ih \tilde{r}\tilde{e}_j(F(\tilde{\Lb}, \tilde{r}\tilde{e}_j))\\
&=2i F(\tilde{\Lb}, \tilde{r}\tilde{e}_j)D_{\tilde{r}\tilde{e}_j}h+ih (D_{\tilde{\Lb}}(\tilde{r}^2\tilde{\rho})+\tilde{r}^2\pa^\mu F(\tilde{\Lb}, \pa_{\mu})).
\end{align*}
In particular we obtain the transport equation 
\begin{align*}
D_{\tilde{\Lb}}(\tilde{r}^2(\lap_A h-i\tilde{\rho} h)) 
&=2i F(\tilde{\Lb}, \tilde{r}\tilde{e}_j)D_{\tilde{r}\tilde{e}_j}h+ih \tilde{r}^2 \pa^\mu F(\tilde{\Lb}, \pa_{\mu}).
\end{align*}
Therefore by using the definition $D_{\tilde{\Lb}} h=0$, we have 
\begin{align*}
\tilde{\Lb}  \langle \tilde{r}^2(\lap_A h-i\tilde{\rho} h), i h \rangle & =  \langle D_{\tilde{\Lb}}(\tilde{r}^2(\lap_A h-i\tilde{\rho} h) ), i h \rangle \\ 
& =  \langle  2i F(\tilde{\Lb}, \tilde{r}\tilde{e}_j)D_{\tilde{r}\tilde{e}_j}h+ih \tilde{r}^2 \pa^\mu F(\tilde{\Lb}, \pa_{\mu}), i h \rangle \\ 
&=    \tilde{r}^2 \pa^\mu F(\tilde{\Lb}, \pa_{\mu}).
\end{align*}
This implies that 
\[
\langle \tilde{r}^2(\lap_A h-i\tilde{\rho} h), i h \rangle=-\int_0^{\tilde{r}}  \tilde{s}^2 \pa^\mu F(\tilde{\Lb}, \pa_{\mu}) ds.
\]
 On the other hand recall that $\langle D_{\tilde{r}\tilde{e}_j} h, h\rangle =0$, it holds that  
\begin{align*}
 \langle \tilde{r}^2(\lap_A h-i\tilde{\rho} h), h \rangle = \langle \tilde{r}^2 \lap_A h , h \rangle
 =\tilde{r}\tilde{e}_j\langle D_{\tilde{r}\tilde{e}_j} h, h\rangle -|D_{\tilde{r}\tilde{e}_j} h|^2= -|D_{\tilde{r}\tilde{e}_j} h|^2.
\end{align*}
Therefore together with the previous identity, we  conclude that 
\begin{align}
\notag 
|\tilde{r}^2(\lap_A h-i\tilde{\rho} h)| & \leq |\int_0^{\tilde{r}}  \tilde{s}^2 \pa^\mu F(\tilde{\Lb}, \pa_{\mu}) ds|+ |D_{\tilde{r}\tilde{e}_1} h|^2+|D_{\tilde{r}\tilde{e}_2} h|^2 \\
\label{eq:lapA:0} 
&\leq |\int_0^{\tilde{r}}  \tilde{s}^2 \pa^\mu F(\tilde{\Lb}, \pa_{\mu}) ds|+ \sum\limits_{j,k} \big( \int_0^{\tilde{r}} F(\tilde{\Lb},  \tilde{\Omega}_{jk}) ds\big)^2.
\end{align}
Basically we need to control the $L^1$ norm of $\tilde{r}^2(\lap_A h-i\tilde{\rho} h)$ on the sphere $S_{(t_0-\tilde{r}, x_0)(\tilde{r})}$. The first term on the right hand side of the above inequality is comparatively easier. We first need to compute the divergence of the Maxwell field $F$. By the definition of $A$ and $F$ (which are independent of $x_3$ and $A_3=0$) and noting that
\begin{align*}
\pa_{\tilde{t}}=\pa_t,\quad \pa_{\tilde{x}_j}=\pa_{x_j}=\pa_j,
\end{align*}
 we have 
\begin{align*}
&\pa^\mu F(\tilde{\Lb}, \pa_\mu)\\
&=-\pa_t F(\pa_t-\tilde{\omega}_j \pa_j, \pa_t)+\pa_1 F(\pa_t-\tilde{\omega}_j \pa_j, \pa_1)+\pa_2 F(\pa_t-\tilde{\omega}_j\pa_j, \pa_2)\\
&= \pa_t (\tilde{\omega}_1 F_{10}+\tilde{\omega}_2 F_{20})+\pa_1(F_{01}-\tilde{\omega}_2 F_{21})+\pa_2 (F_{02}-\tilde{\omega}_1 F_{12})\\
&=(\pa_1-\tilde{\omega}_1 \pa_t)\Im(\phi\cdot \overline{D_2\phi})+(\tilde{\omega}_2\pa_1-\tilde{\omega}_1\pa_2) \Im(\phi\cdot \overline{D^0\phi})+(\tilde{\omega}_2\pa_t-\pa_2) \Im(\phi\cdot \overline{D_1\phi})\\
&=2\Im(D_1\phi\cdot \overline{D_2\phi})+\Im(\phi \cdot\overline{[D_1, D_2]\phi})+2\tilde{\omega}_1 \Im(D_2\phi\cdot \overline{D_0\phi})+\tilde{\omega}_1 \Im(\phi \cdot\overline{[D_2, D_0]\phi}) \\
&\quad +2\tilde{\omega}_2 \Im(D_0\phi\cdot \overline{D_1\phi})+\tilde{\omega}_2 \Im(\phi \cdot\overline{[D_0, D_1]\phi})\\
&= \Im(D_1\phi\cdot \overline{D_2\phi}-D_2\phi\cdot \overline{D_1\phi})-(F_{12}+\tilde{\omega}_1 F_{20}+\tilde{\omega}_2 F_{01})|\phi|^2   +2 \Im(D_0\phi\cdot \overline{ (\tilde{\omega}_2 D_1\phi-\tilde{\omega}_1 D_2\phi ) }) \\
&=2\Im(\tilde{\D}_1\phi\cdot \overline{\tilde{\D}_2\phi})- \Im(\phi\cdot \overline{( \tilde{\omega}_1 D_1\phi +\tilde{\omega}_2 D_2\phi -D_0\phi  )})  |\phi|^2    \\
&\quad +2\Im(D_{\tilde{r}}\phi \cdot \overline{(\tilde{\omega}_1\tilde{\D}_2\phi -\tilde{\omega}_2 \tilde{\D}_1\phi  ) })+2 \Im(D_0\phi\cdot \overline{(\tilde{\omega}_2 D_1\phi-\tilde{\omega}_1 D_2\phi )}).
\end{align*}
Here $\tilde{\D}_j=D_{\tilde{x}_j}-\tilde{\omega}_j D_{\tilde{r}}$. Since $\phi$ is independent of $x_3$ and $\pa_{\tilde{x}_\mu}=\pa_{x_\mu }$, we have 
\begin{align*}
&\tilde{\omega}_1 D_1\phi +\tilde{\omega}_2 D_2\phi -D_0\phi=\tilde{\omega}_j \tilde{D}_j\phi  -\tilde{D}_0\phi=-D_{\tilde{\Lb}}\phi,\\
& \tilde{\omega}_1\tilde{\D}_2\phi -\tilde{\omega}_2 \tilde{\D}_1\phi = \tilde{\omega}_1\tilde{D}_2\phi -\tilde{\omega}_2 \tilde{D}_1\phi =-(\tilde{\omega}_2 D_1\phi-\tilde{\omega}_1 D_2\phi )).
\end{align*}
Therefore we can write that 
\begin{align*}
\Im(D_{\tilde{r}}\phi \cdot \overline{(\tilde{\omega}_1\tilde{\D}_2\phi -\tilde{\omega}_2 \tilde{\D}_1\phi  ) })+ \Im(D_0\phi\cdot \overline{(\tilde{\omega}_2 D_1\phi-\tilde{\omega}_1 D_2\phi )})= -\Im(D_{\tilde{\Lb}}\phi \cdot \overline{(\tilde{\omega}_1\tilde{\D}_2\phi -\tilde{\omega}_2 \tilde{\D}_1\phi  ) }).
\end{align*}
We thus can bound that 
\begin{align*}
|\pa^\mu F(\tilde{\Lb}, \pa_\mu )|\les  |\phi|^6+ |D_{\tilde{\Lb}}\phi|^2+|\tilde{\D}\phi|^2.
\end{align*}
Using Sobolev embedding on the surface $\mathcal{N}^{-}(q)$ together with the local energy decay of Proposition \ref{prop:EF:cone:gamma}, we can estimate that 
\begin{align}
\notag
\int_{0}^{\tilde{r}}\int_{\tilde{\omega}} s^2|\pa^\mu F(\tilde{\Lb}, \pa_\mu)|ds d\tilde{\omega} &\les \int_{0} ^{\tilde{r}}\int_{\tilde{\omega}} (|\phi|^6+ |D_{\tilde{\Lb}}\phi|^2+|\tilde{\D}\phi|^2)s^2 ds d\tilde{\omega}\\
\label{eq:lapA:dF}
&\les   \int_{ \mathcal{N}^{-}(q)\cap\{t_0-\tilde{r}\leq t\leq t_0\}} ( |D_{\tilde{\Lb}}\phi|^2+|\tilde{\D}\phi|^2+ \Lambda^{\frac{5-p}{2}}    |\phi |^{p+1})s^2 ds d\tilde{\omega} \\
\notag
&\les \sqrt{\tilde{r}}. 
\end{align}
Next for the second term on the right hand side of \eqref{eq:lapA:0}, we  compute that 
\begin{align*}
F(\tilde{\Lb}, \tilde{\Omega}_{12})&=F(\pa_t-\tilde{\omega}_j \pa_j, \tilde{x}_1\pa_2-\tilde{x}_2\pa_1)=\tilde{x}_1 F_{02}-\tilde{x}_2 F_{01}-\tilde{\omega}_j\tilde{x}_1F_{j2}+\tilde{\omega}_j\tilde{x}_2F_{j1}\\
&=\tilde{r}(\tilde{\omega}_1^2+\tilde{\omega}_2^2)\Im(\phi\cdot \overline{D_{\tilde{t}}\phi})-\tilde{x}_1\Im(\phi\cdot \overline{\tilde{\omega}_1 D_{\tilde{r}}\phi +\tilde{\D}_1\phi  })-\tilde{x}_2\Im(\phi\cdot \overline{\tilde{\omega}_2 D_{\tilde{r}}\phi +\tilde{\D}_2\phi  }) \\ 
&= \tilde{r}(\tilde{\omega}_1^2+\tilde{\omega}_2^2)\Im(\phi\cdot \overline{D_{\tilde{\Lb}}\phi})-\tilde{x}_1\Im(\phi\cdot \overline{ \tilde{\D}_1\phi  })-\tilde{x}_2\Im(\phi\cdot \overline{ \tilde{\D}_2\phi  }),\\
&=\tilde{r} \Im(\phi\cdot \overline{D_{\tilde{\Lb}}\phi})-\tilde{r}\tilde{\omega}_3 (\Im(\phi\cdot \overline{ \tilde{\D}_3\phi  }) +\tilde{\omega}_3 \Im(\phi\cdot \overline{D_{\tilde{\Lb}}\phi})),\\
F(\tilde{\Lb}, \tilde{\Omega}_{k3})&=F(\pa_t-\tilde{\omega}_j \pa_j, \tilde{x}_k\pa_3-\tilde{x}_3\pa_k)= -\tilde{x}_3 F_{0k} +\tilde{\omega}_j\tilde{x}_3F_{jk}\\
&= \tilde{x}_3 \varepsilon_{0jk}(\Im(\phi\cdot \overline{D_{j}\phi })-\tilde{\omega}_j\Im(\phi\cdot \overline{D_{\tilde{t}}\phi}))\\ 
&= \tilde{r} \tilde{\omega}_3\varepsilon_{0jk} (  \Im(\phi\cdot \overline{ \tilde{\D}_j \phi  }) -\tilde{\omega}_j\Im(\phi\cdot \overline{D_{\tilde{\Lb}}\phi})   ).
\end{align*}
In particular we can bound that 
\begin{align}
\label{eq:lapA:Fr}
|F(\tilde{\Lb}, \tilde{\Omega}_{12})- \tilde{r} \Im(\phi\cdot \overline{D_{\tilde{\Lb}}\phi})|+|F(\tilde{\Lb}, \tilde{\Omega}_{k3})| \les |\tilde{\omega}_3|\tilde{r}(|D_{\tilde{\Lb}}\phi|+|\tilde{\D}\phi|)|\phi|. 
\end{align}
The idea now is that we rely on the rough pointwise decay estimate \eqref{eq:pt:phi:cf:3d} and control the tangential derivative by using the energy flux of Proposition \ref{prop:EF:cone:gamma} for the solution $\phi$. The crucial observation now is that $\tilde{\omega}_3$ provides the necessary smallness which cancels the possible blow up of $\phi$ on the boundary. In fact under spherical coordinates, we can write that 
\begin{align*}
x_1=x_{01}+s\sin\theta_1 \cos\theta_2,\quad x_2=x_{02}+s\sin\theta_1 \sin\theta_2,\quad x_3=x_{03}+s\cos\theta_1. 
\end{align*}
In view of the rough pointwise decay estimate \eqref{eq:pt:phi:cf:3d} for $\phi$, for $\gamma=\frac{1}{2}-\frac{\ep_1}{4}$, we bound that 
\begin{align*}
&\int_0^{\tilde{r}} \tilde{\omega}_3^2 (1-t )^{\gamma } |\phi(t_0-s, x)|^2 ds \\ 
 &\les  \cos^2\theta_1 \int_0^{\tilde{r}}    (1-t-r_1)^{\frac{p-5-\ep_1}{4}} (1-t)^{-1+\gamma } ds \\
& \les  \cos^2\theta_1 \tilde{r}^{-1+\gamma}\int_0^{\tilde{r}}    (s -s|\sin\theta_1|)^{\frac{p-5-\ep_1}{4}}  ds\\ 
&\les \tilde{r}^{-1+\gamma+\frac{p-1-\ep_1}{4}}\\ 
&\les \tilde{r}^{ \frac{p-3-2\ep_1}{4}}.
\end{align*} 
Here by the definition, $t=t_0-s$, $r_1=\sqrt{x_1^2+x_2^2}$, $r_{01}=\sqrt{x_{01}^2+x_{02}^2}$. 
Therefore by using the weighted energy flux bound of Proposition \ref{prop:EF:cone:gamma} we can show that 
\begin{align}
\notag
&\int_{S_{(t_0-\tilde{r}, x_0)(\tilde{r})}} \Big(  \int_0^{\tilde{r}} |\tilde{\omega}_3|s(|D_{\tilde{\Lb}}\phi|+|\tilde{\D}\phi|)|\phi| ds \Big)^2 d\tilde{\omega}\\
\notag
 & \leq \int_{\tilde{\omega}}    \int_0^{\tilde{r}} s^2(|D_{\tilde{\Lb}}\phi|^2+|\tilde{\D}\phi|^2) (1-t)^{-\gamma }ds    \cdot \int_0^{\tilde{r}}\tilde{\omega}_3^2 (1-t )^{\gamma } |\phi(t, x)|^2 ds   d\tilde{\omega} \\ 
 \label{eq:lap:dl}
& \les \tilde{r}^{ \frac{p-3-2\ep_1}{4}} \int_{ \mathcal{N}^{-}(q)} (|D_{\tilde{\Lb}}\phi|^2+|\tilde{\D}\phi|^2) (1-t)^{-\gamma } s^2 ds d\tilde{\omega} \\ 
\notag
&\les \tilde{r}^{ \frac{p-3-2\ep_1}{4}}.
\end{align}
Finally in view of \eqref{eq:lapA:Fr}, it remains to control
\[
\int_{ S_{(t_0-\tilde{r}, x_0)(\tilde{r})} } \left(\int_0^{\tilde{r}}s \Im(\phi\cdot \overline{D_{\tilde{\Lb}}\phi })ds\right)^2 d\tilde{\omega}.
\]
The problem for this term is that there is no smallness we can use. However, the special structure allows us to apply the sharp geometric trace theorem developed by Klainerman-Rodnianski in \cite{Klainerman06:trace}.
\begin{Thm}[Geometric bilinear trace theorem]
\label{them:Gbilinear}
Let $\mathcal{H}$ be the conic cone $\{(t, x)| t+|x|=1, 0\leq t\}$ and $A$ a 1-form in $\R^{1+3}$. Then there exists a constant $C$ such that 
\begin{align*}
\int_{|\omega|=1} \Big(\int_0^1 (f\cdot D_{\Lb} g)(s, s\omega) ds\Big)^2  d\omega \leq C  \mathcal{N}_1(f) \cdot \mathcal{N}_1(g)
\end{align*}
for any sufficiently  smooth function $f, g$ on $\R^{1+3}$. 
Here the energy $\mathcal{N}_1(f)$ is defined for any $0\leq \tau\leq 1$ 
\[
\mathcal{N}_{\tau}(f)= \int_0^{\tau}  \int_{|\omega|=1} \tau(|D_{\Lb} f|^2 +|\D f|^2) +\tau^{-1}|f|^2 ds d\omega 
\]
under the null frame $\{L=\pa_t+\pa_r,\quad \Lb=\pa_t-\pa_r, e_1, e_2\}$ and $D$ is the covariant derivative associated to the connection field $A$. 
\end{Thm}
See Theorem 4.1 in   \cite{Klainerman06:trace}. We emphasize here that the original version of Theorem 4.1 in  \cite{Klainerman06:trace} is much stronger and applies to the more general setting of Lorentzian manifolds. The connection field $A$ plays the geometric  role of a Lorentzian metric.

To apply the above theorem, consider the cone $\mathcal{N}^{-}(q)\cap \{t_0-\tilde{r}\leq t\leq t_0\}$. Scaling implies that 
 \begin{align*}
\int_{|\omega|=1} \Big(\int_0^{\tilde{r}} (f\cdot D_{\Lb} g)(s, s\omega) ds\Big)^2  d\omega \leq C  \mathcal{N}_{\tilde{r}}(f) \cdot \mathcal{N}_{\tilde{r}}(g).
\end{align*}
Note that 
\[
s^2 \Im(\phi\cdot \overline{D_{\tilde{\Lb}}\phi})=\Im( (s\phi) \cdot  \overline{D_{\tilde{\Lb}}(s\phi)}).
\]
In view of Proposition \ref{prop:EF:cone:gamma}, we can bound that 
\begin{align*}
\mathcal{N}_{\tilde{r}}(s\phi)&=\int_0^{\tilde{r}}\int_{|\tilde{\omega}|=1} \tilde{r}(|D_{\tilde{\Lb}}(s\phi)|^2+|\tilde{\D}(s\phi)|^2 )+\tilde{r}^{-1} |s\phi|^2 dsd\tilde{\omega}\\ 
&=\tilde{r}\int_0^{\tilde{r}}\int_{|\tilde{\omega}|=1} (|D_{\tilde{\Lb}}\phi|^2+|\tilde{\D} \phi|^2)s^2 dsd\tilde{\omega}+ \tilde{r}^{-1}\int_0^{\tilde{r}}\int_{|\tilde{\omega}|=1} |\phi|^2 s^2 dsd\tilde{\omega}+\tilde{r}^2\int_{|\tilde{\omega}|=1}  |\phi|^2 d\tilde{\omega}|_{s=\tilde{r}}\\ 
&\leq \tilde{r}\int_0^{\tilde{r}}\int_{|\tilde{\omega}|=1} (|D_{\tilde{\Lb}}\phi|^2+|\tilde{\D} \phi|^2+|\phi|^{p+1})s^2 dsd\tilde{\omega}+ \tilde{r} \int_0^{\tilde{r}}\int_{|\tilde{\omega}|=1} |D_{\tilde{\Lb}}\phi|^2 s^2 dsd\tilde{\omega}\\ 
&\les \tilde{r}\sqrt{\tilde{r}}.
\end{align*}
The second last step follows from a type of Hardy's inequality by controlling $\phi$ with its derivative $|D_{\tilde{\Lb}}\phi|$ and the value on the sphere $s=\tilde{r}$. 

Now denote 
\[
G(\tilde{r}, \tilde{\omega})=\int_0^{\tilde{r}} \Im((s\phi)\cdot \overline{D_{\tilde{\Lb}}(s\phi)}) ds. 
\]
In particular by using the rough bound \eqref{eq:pt:phi:cf:3d}, we have 
\begin{align*}
\lim\limits_{\tilde{r}\rightarrow 0} \tilde{r}^{-1}G(\tilde{r}, \tilde{\omega})=0.
\end{align*}
Then the above argument implies that 
\[
\|G(\tilde{r}, \tilde{\omega})\|_{L^2_{\tilde{\omega}}}\les \tilde{r}^{\frac{3}{2}}. 
\]
Taking the derivative in in terms of $\tilde{r}$, we derive that 
\[
s^{-1}\pa_{s} G(s, \tilde{\omega})=s^{-1} \Im((s\phi)\cdot \overline{D_{\tilde{\Lb}}(s\phi)})=s \Im(\phi\cdot \overline{D_{\tilde{\Lb}}\phi}).
\]
Integrate in terms of $s$ from $0$ to $\tilde{r}$. We can estimate that 
\begin{align*}
\int_0^{\tilde{r}}s \Im(\phi\cdot \overline{D_{\tilde{\Lb}}\phi}) ds= \int_0^{\tilde{r}}  s^{-1}\pa_{s} G(s, \tilde{\omega}) ds = \tilde{r}^{-1}G(\tilde{r}, \tilde{\omega})+\int_0^{\tilde{r}} G (s, \tilde{\omega}) s^{-2}ds.
\end{align*}
We therefore can show that 
\begin{align*}
  \| \int_0^{\tilde{r}}s \Im(\phi\cdot \overline{D_{\tilde{\Lb}}\phi}) ds \|_{L^2_{\tilde{\omega}}} &\leq  \| \tilde{r}^{-1}G(\tilde{r}, \tilde{\omega}) \|_{L^2_{\tilde{\omega}}}+\|\int_0^{\tilde{r}} s^{-2}   G(s, \tilde{\omega}) ds \|_{L^2_{\tilde{\omega}}}\\ 
  &\les \tilde{r}^{-1}\tilde{r}^{\frac{3}{2}}  +\int_0^{\tilde{r}} s^{-2}  \| G(s, \tilde{\omega}) \|_{L^2_{\tilde{\omega}}} ds \\ 
  &\les  \tilde{r}^{\frac{1}{2}}  +\int_0^{\tilde{r}} s^{-2}  s^{\frac{3}{2}} ds \\
  &\les \sqrt{\tilde{r}}.
\end{align*}
Combining the estimates \eqref{eq:lapA:0}, \eqref{eq:lapA:dF}, \eqref{eq:lapA:Fr} and \eqref{eq:lap:dl}, we have shown that 
\begin{align*}
\int_{ S_{(t_0-\tilde{r}, x_0)(\tilde{r})} }|\tilde{r}^2(\lap_A h-i\tilde{\rho} h)| d\tilde{\omega}  \les \sqrt{\tilde{r}}+ \tilde{r}^{ \frac{p-3-2\ep_1}{4}}+\tilde{r}\les \tilde{r}^{\frac{1}{4}}.
\end{align*}
Here recall that $p>4$ and $p-3-2\ep_1>1$. 

Note that on the cone $\mathcal{N}^{-}(q)$ it holds that 
\begin{align*}
1-t_0+|x_0| \leq  1-(t_0-s)+|x_0+s\tilde{\omega}|=1-t+r. 
\end{align*}
 From the representation formula \eqref{eq:rep4phi} and the bound \eqref{eq:spt:pterm}, we then obtain that 
\begin{equation*}
\begin{split}
&(1-t_0+|x_0|)^{\frac{5+2\ep -p}{2}}|\phi(t_0, x_0) |^{p_1'} \\ 
 &\les 1 + \Big|\int_{0}^{t_0}\int_{|\tilde{\omega}|=1}|\widetilde{\lap}_A h-i\widetilde{\rho}| \mathcal{M}(t_0-\tilde{r})^{\frac{1}{p_1'}}   \widetilde{r} d\widetilde{r}d\widetilde{\om} \Big|^{p_1'} +\int_0^{t_0}\mathcal{M}(t_0-\tilde{r}) \tilde{r}^{-1+\ep }d\tilde{r}\\ 
 &\les 1 + \Big|\int_{0}^{t_0}  \mathcal{M}(t_0-\tilde{r})^{\frac{1}{p_1'}}   \widetilde{r}^{-\frac{3}{4}} d\widetilde{r} \Big|^{p_1'} +\int_0^{t_0}\mathcal{M}(t_0-\tilde{r}) \tilde{r}^{-1+\ep }d\tilde{r}\\ 
 &\les 1 +  \int_0^{t_0}\mathcal{M}(t_0-\tilde{r}) \tilde{r}^{-1+\ep }d\tilde{r}.
\end{split}
\end{equation*}
For fixed $t_0$, take supreme for $|x_0|\leq 1-t_0$. By the definition of $\mathcal{M}(t)$, we then obtain that 
\[
\mathcal{M}(t_0)\les 1+\int_0^{t_0}\mathcal{M}(t_0-\tilde{r}) \tilde{r}^{-1+\ep }d\tilde{r},
\]
 which implies the uniform bound 
 \[
\mathcal{M}(t_0)\les 1,\quad \forall t_0\leq 1
 \]
 by using the following variant of 
 Gronwall's inequality.
\begin{Lem}
Let $f$ be a positive continuous function on $[0, 1]$. If there exist positive constants $A$, $B$ and $\gamma\in(0, 1)$ so that for all $t_0\in [0, 1]$ it holds that 
\begin{align*}
f(t_0)\leq A+B\int_0^{t_0} s^{-\gamma} f(t_0-s)ds.
\end{align*}
Then, there exists a constant $C$ depending only on $B$ and  $\gamma$ so that
\begin{align*}
f(t)\leq C A,\quad \forall 0\leq t\leq 1. 
\end{align*}
\end{Lem}
See Lemma 2.2 in \cite{YangYu:MKGrough}.

\subsection{Proof for the main theorem \ref{thm:main:pd:imp}}    
 The previous section shows that the solution $\phi$ to the system \eqref{eq:CSH:3d} verifies the bound 
 \begin{align*}
 |\phi(t, x)|\les (1-t+r)^{ -\frac{5+2\ep -p}{2p_1'}} \mathcal{M}(t)^{\frac{1}{p_1'}}\les (1-t)^{- \frac{5+2\ep -p}{2p_1'}}\les (1-t)^{-\frac{5-p+\ep_2}{p+1}}
 \end{align*}
 for some small positive constant $\ep_2$ (depending on $p$, $\ep$ and $\ep_1$).  
By the definition of $\phi(t, x)$ and $A$ in $\R^{1+3}$, the solution $\tilde{\phi}(\tilde{ t}, \tilde{x}_1, \tilde{x}_2)$ to the equation \eqref{eq:CS:cf} also verifies the bound 
 \begin{align*}
 |\tilde{\phi}(\tilde{t}, \tilde{x}_1, \tilde{x}_2)|=|\phi|  \les (1-\tilde{t})^{-\frac{5-p+\ep_2}{p+1}}. 
 \end{align*}
 Finally in view of the conformal transformation, the original solution $\phi(t, x_1, x_2)$ (by abusing the notation, the original solution $\phi$ is different to $\phi(t, x_1, x_2, x_3)$ defined in $\R^{1+3}$) to \eqref{eq:CSH:p:2d} in $\R^{1+2}$ then satisfies the pointwise bound
\begin{align*}
|\phi(t, x_1, x_2)| & \les \La^{-\frac{1}{2}} |\tilde{\phi}(\tilde{t}, \tilde{x}_1, \tilde{x}_2)|\les (2+t+r_1)^{-\frac{1}{2}}(2+t-r_1)^{-\frac{1}{2}} (2+t-r_1)^{\frac{5-p+\ep_2}{p+1}} \\ 
&\les (2+t+r_1)^{-\frac{1}{2}} (2+|t-r_1|)^{ -\frac{3p-9-2\ep_2}{2(p+1)}}.
\end{align*} 
Here $r_1=\sqrt{x_1^2+x_2^2}$ and the solution $\phi$ vanishes out side the cone $t\leq r_1-1$ since the initial data are compactly supported. We thus finished the proof for Theorem \ref{thm:main:pd:imp}.

\bibliographystyle{plain}

\end{document}